\documentclass[12pt, eqno]{article}
\usepackage{amsmath,amsthm,amscd,amssymb}
\usepackage{latexsym}
\usepackage{extarrows}
\usepackage{color}
\usepackage{color}
\usepackage[T1]{fontenc}
\usepackage[utf8]{inputenc}
\usepackage{authblk}
\newtheorem{lemma}{Lemma}[section]
\newtheorem{remark}{Remark}[section]
\newcommand{\bremark}{\begin{remark} \em}
\newcommand{\eremark}{\end{remark} }
\makeatletter
\newcommand{\rmnum}[1]{\romannumeral #1}
\newcommand{\Rmnum}[1]
{\expandafter\@slowromancap\romannumeral #1@}
\makeatother
\newtheorem{theorem}{Theorem}[section]

\newtheorem{definition}[theorem]{Definition}

\begin{document}
\begin{titlepage}
\title{
Global existence of spherically symmetry solutions for isothermal Euler-Poisson system outside a ball}
\author[]{
Lingjun Liu }
\affil{\footnotesize College of Mathematics, Faculty of Science, Beijing University of Technology, Beijing 100124, P.R.China. \\
	E-mail: lingjunliu@bjut.edu.cn.}
\date{}
\end{titlepage}
\maketitle
\vspace{-3em}

\abstract{\small In this paper, 
	we consider an isothermal Euler-Poisson system with self-gravitational force, modeling a compact star such as strange quark star. We prove that there exists  a global entropy solution with spherically symmetry outside a ball, through the fractional Lax-Friedrichs scheme and the theory of compensated compactness. 

\


\section{Introduction}\label{section1}

\setcounter{section}{1}

\setcounter{equation}{0}

In this paper, we are concerned with an Euler-Poisson system for compressible fluids with self-gravitational force in multi-dimension ($N\ge 2$),  
\begin{equation}\label{eq111}
\left\{
\begin{array}{ll}
\displaystyle\partial_t\rho+\nabla\cdot\overrightarrow{m}=0, \quad \quad \overrightarrow{x}\in\mathbb{R}^N,\quad t\geq0,\\[1mm]
\displaystyle\partial_t\overrightarrow{m}+\nabla\cdot\frac{\overrightarrow{m}\otimes \overrightarrow{m}}{\rho}+\nabla p(\rho)=\rho\nabla\Phi,\\
\displaystyle-\Delta\Phi=\rho,
\end{array}
\right.
\end{equation}
where $\rho, \overrightarrow{m}, \Phi$ denote the dust density, the momentum, the gravitational potential generated by the self-gravitation of the fluids respectively. $p(\rho)$ is the equation of state. In the case of compact star \cite{w1972}, $p(\rho)$ is a linear and increasing function of $\rho$ and the spatial dimension is 3, i.e., $N=3$. Without loss of generality, let $p(\rho)=\rho^\gamma$, $\gamma=1$. And this corresponds to the pressure state of isothermal gas for the compressible Euler equations. We consider the spherically symmetric solutions to the system \eqref{eq111} with the form:
\begin{equation}\label{eq112}
\begin{aligned}
(\rho, \overrightarrow{m}, \Phi)(\overrightarrow{x}, t)=(\rho(x, t), m(x, t)\frac{\overrightarrow{x}}{x}, \Phi(x, t)),\quad  x=|\overrightarrow{x}|. \\
\end{aligned}
\end{equation}
Then $(\rho, m, \Phi)(x, t)$ in \eqref{eq112} is governed by the one-dimensional Euler-Poisson system with geometric source terms:
\begin{equation}\label{eq113}
\left\{
\begin{array}{ll}
\displaystyle\partial_t\rho+\partial_xm=-\frac{N-1}{x}m, \\
\displaystyle\partial_tm+\partial_x(\frac{m^2}{\rho}+\rho)=-\frac{N-1}{x}\frac{m^2}{\rho}+\rho\partial_x\Phi,\\[3mm]
\displaystyle-\partial_{xx}\Phi=\frac{N-1}{x}\partial_x\Phi+\rho.
\end{array}
\right.
\end{equation}

In this paper, we focus on the system \eqref{eq113} outside a unit ball in $\mathbb{R}^N$, i.e. in the spatial region $\Pi=\{\overrightarrow{x}\in\mathbb{R}^N, x\geq1\}$.
By the Poisson equation of \eqref{eq113}, we have
\begin{equation}\label{eq117}
\begin{aligned}
\Phi_x=-\frac{1}{x^{N-1}}\int_1^xs^{N-1}\rho(s, t)ds , \quad x\geq1 
\end{aligned}
\end{equation}
with boundary conditions
\begin{equation}\label{eq115}
\left\{
\begin{array}{ll}
\Phi_x\big|_{\partial\Pi}=0, \quad \ \mbox{a.e.} \quad t>0, \\
\Phi(x, t)\rightarrow0,\quad  \ \mbox{as} \quad x\rightarrow+\infty, \quad \ \mbox{a.e.} \quad t>0.\\
\end{array}
\right.
\end{equation}
Then the system \eqref{eq113} becomes
\begin{equation}\label{eq118}
\left\{
\begin{array}{ll}
\displaystyle\partial_t\rho+\partial_xm=-\frac{N-1}{x}m, \\
\partial_tm+\partial_x(\frac{m^2}{\rho}+\rho)=-\frac{N-1}{x}\displaystyle\frac{m^2}{\rho}-\frac{\rho}{x^{N-1}}\int_1^xs^{N-1}\rho(s, t)ds,\\
\end{array}
\right.
\end{equation}
with an initial-boundary conditions
\begin{equation}\label{eq114}
\left\{
\begin{array}{ll}
(\rho, m)(x,t=0)=(\rho_0(x),m_0(x)),\\
m\big|_{\partial\Pi}=0, \qquad \ \mbox{a.e.} \quad t>0.
\end{array}
\right.
\end{equation}
The conserved total mass is
\begin{equation}\label{eq116}
\begin{aligned}
M:=\iint\limits_{\Pi} \rho(\overrightarrow{x},t)d\overrightarrow{x}=\int_1^{+\infty}\omega_nx^{N-1}\rho(x,t)dx=\int_1^{+\infty}\omega_nx^{N-1}\rho_0(x)dx,
\end{aligned}
\end{equation}
 where $\omega_n$ is the volume of unit ball in ${\mathbb{R}^N}$ and $N\geq2$.

To treat the terms arising from the geometric source, it is more convenient to introduce  weight density $\varrho$ and momentum $\omega $ as follows, 
\begin{equation}\label{eq119}
\begin{aligned}
\varrho=x^{N-1}\rho, \quad \omega=x^{N-1}m,
\end{aligned}
\end{equation}
satisfying
\begin{equation}\label{eq120}
\left\{
\begin{array}{ll}
\displaystyle\partial_t\varrho+\partial_x\omega=0, \\
\displaystyle\partial_t\omega+\partial_x(\frac{\omega^2}{\varrho}+\varrho)=\frac{N-1}{x}\varrho-\frac{\varrho}{x^{N-1}}\int_1^x\varrho(s, t)ds,\\
\end{array}
\right.
\end{equation}
with 
\begin{equation}\label{eq1114}
\left\{
\begin{array}{ll}
(\varrho, \omega)(x,t=0)=(\varrho_0(x),\omega_0(x)),\\
\omega\big|_{\partial\Pi}=0, \qquad \ \mbox{a.e.} \quad t>0.
\end{array}
\right.
\end{equation}

Since the system \eqref{eq111} is a multi-dimensional hyperbolic system of conservation laws with self-gravitational force, the smooth solution may break down in finite time. Therefore one has to investigate the weak entropy solutions with general initial date. When the gravitational potential disappears, the system \eqref{eq111} is a compressible Euler system
\begin{equation}\label{eqd111}
    \left\{
    \begin{array}{ll}
 \displaystyle   \partial_t\rho+\nabla\cdot\overrightarrow{m}=0,\\[3mm]
  \displaystyle  \partial_t\overrightarrow{m}+\nabla\cdot\frac{\overrightarrow{m}\otimes \overrightarrow{m}}{\rho}+\nabla p(\rho)=0,
    \end{array}
    \right.
    \end{equation}
    where the shock is usually generated in finite time no matter how the initial values smooth, see \cite{smoller}.

Various efforts have been made on the existence of weak entropy solutions to the Euler-Poisson system \eqref{eq111}  and Euler system \eqref{eqd111}. For the one-dimensional Euler system \eqref{eqd111} with the initial data staying away from vacuum, the global existence of BV solution for isothermal case (i.e. $ \gamma=1$) was first established by Nishida \cite{N1},
and the case for $\gamma>1$ was later proved by Nishida-Smoller \cite{NS}. When the solution contains vacuum, DiPerna \cite{D} gave the first global existence result for $\gamma=1+\frac{2}{2N+1}\ ( N\ge2)$ by the theory of compensated compactness due to Murat \cite{M12} and Tartar \cite{T11}. Then the global existence of $L^\infty$ weak entropy solution was shown by Ding, Chen and Luo \cite{C1, DCL1, DCL2} via the Lax-Friedrichs scheme 
for $1<\gamma\leq\frac{5}{3}$. Whereafter, global existence results for $\gamma>1$  were obtained in \cite{LPS, LPT} through the kinetic formulation. The case for $ \gamma=1$ with vacuum was achieved by Huang-Wang in \cite{HW}.

The global existence of solutions to the multidimensional Euler system \eqref{eqd111} is a huge challenging open problem. There are very few works except for the solutions with spherically symmetry. When $\gamma=1$, the global existence of BV solutions outside a ball was given by Makino-Mizohata-Ukai \cite{MMU1, MMU2} via Glimm scheme. When $\gamma\in(1,\frac{5}{3}]$, there have some interesting works related to the weak entropy solutions with spherically symmetry in \cite{C2, CG, CP, HLY1, LW, MT, T} with the help of the theory of compensated compactness. 

Now let's move to 
the Euler-Poisson system \eqref{eq111}. 
 For the $3-D$ gaseous stars, a compactly supported expanding classical solution was discovered by Goldreich-Weber \cite{GW}, see also \cite{FL, M2}. Hadzic-Jang \cite{HJ1} further proved the nonlinear stability of the Goldreich-Weber solution 
 with $\gamma=\frac{4}{3}$. 
 In addition, Hadzic-Jang \cite{HJ2} constructed a class of global-in-time solutions of the $3-D$ Euler-Poisson equations 
 for $\gamma=1+\frac{1}{k}, k\in\mathbb{N}\setminus{\{1\}}$ or $\gamma\in(1, \frac{14}{13})$. More recently, Guo-Hadzic-Jang \cite{GHJ} constructed a 
family of collapsing solutions of Euler-Poisson system \eqref{eq111}. 
For general initial data, one has to seek for the weak entropy solutions. 
Makino \cite{M} and 
 Xiao \cite{X} got the local and global existence of weak entropy solutions with spherically symmetry for the gaseous star model \eqref{eq111} with $\gamma\in(1,\frac{5}{3}]$ outside a ball respectively. 
    
In this paper, we 
aim to obtain the global existence of weak entropy solutions with spherically symmetry 
 for the isothermal Euler-Poisson system \eqref{eq111} outside a ball.

Below we give the definition of weak entropy solution.
\vskip 0.1in
\begin{definition}\label{def2} A pair of mappings $(\eta, q): \mathbb{R}^2\rightarrow\mathbb{R}^2$ are said to be entropy pair of system \eqref{eq124}, if they satisfy
\begin{equation}\label{eq130}
\begin{aligned}
\bigtriangledown q=\bigtriangledown\eta\bigtriangledown f.
\end{aligned}
\end{equation}
If $\eta|_{\varrho=0}=0$ and $\bigtriangledown^2\eta\geq0$, $\eta$ is called a weak and convex entropy.
\end{definition}
\vskip 0.1in


It is well known that the mechanic entropy flux pair
\begin{equation}\label{eq13391}
 \begin{aligned}
\eta_e=\frac{1}{2}\frac{m^2}{\rho}+\rho\log\rho,\ \ q_e=\frac{1}{2}\frac{m^3}{\rho^2}+m+m\log\rho
\end{aligned}
\end{equation}
is weak and convex. 


\vskip 0.1in
    \begin{definition}[\bf Weak entropy solution]\label{def1}  
       For any $T>0$, the function $(\rho, m)(x,t)\in L^\infty(\Pi_T)$, $\Pi_T:=[1,+\infty)\times[0,T)$, is called a global weak entropy solution of \eqref{eq118} $-$ \eqref{eq114} in the region $\Pi_T$ provided that, it holds for any test function $\varphi\in C_0^\infty(\Pi_T)$, 
    \begin{equation}\label{eq140}
    \left\{
    \begin{array}{ll}
\displaystyle    \int_{0}^{+\infty}\int_1^{+\infty}(\rho\varphi_t+m\varphi_x-\frac{N-1}{x}m\varphi )dxdt+\int_1^{+\infty}\rho_0(x)\varphi(x,0)dx=0,  \\[3mm]
\displaystyle    \int_{0}^{+\infty}\int_1^{+\infty}\big[m\varphi_t+(\frac{m^2}{\rho}+\rho)\varphi_x-(\frac{N-1}{x}\frac{m^2}{\rho}+\frac{\rho}{x^{N-1}}\int_1^xy^{N-1}\rho(y,t)dy)\varphi\big]dxdt\\
    \qquad  \qquad\qquad \qquad\displaystyle+\int_1^{+\infty}m_0(x)\varphi(x,0)dx=0,
    \end{array}
    \right.
    \end{equation}
    and for any non-negative test function $\psi\in C_0^\infty(\Pi_T)$, 
    \begin{equation}\label{eq008}
    \begin{array}{ll}
    \displaystyle   \int_{0}^{+\infty}\int_1^{+\infty}\Big[\eta_e(\rho,m)\psi_t+q_e(\rho,m)\psi_x-\frac{N-1}{x}m\eta_{e\rho}(\rho,m)\psi-\Big(\frac{N-1}{x}\frac{m^2}{\rho}\\ 
    \displaystyle+\frac{\rho}{x^{N-1}}\int_1^xy^{N-1}\rho(y,t)dy\Big)\eta_{em}(\rho,m)\psi\Big]dxdt+\int_1^{+\infty}\eta_e(\rho_0,m_0)\psi(x,0)dx\geq0,
       \end{array}
    \end{equation}
    for the mechanical entropy $(\eta_e,q_e)(\rho,m)$ given in  \eqref{eq13391}, and when $\varepsilon\rightarrow0$,
    \begin{equation}\label{eq08}
		\begin{aligned}
            \frac{1}{\varepsilon}\int_1^{1+\varepsilon}m(x,t)dx
            {\rightarrow}0 \quad \ \ \mbox{in} \quad L^\infty_{loc}((0,\infty)) \ \mbox{weak}^*.
        \end{aligned}
	\end{equation}      
\end{definition}
    \vskip 0.1in


Our main result is given as follows.
\vskip 0.1in

\begin{theorem}\label{theorem2} Assume that the initial data satisfies
\begin{equation}\label{eq141}
\begin{array}{ll}
0\leq\rho_0(x)\leq \frac{C_0}{x^{N-1}}<\infty, \qquad |m_0(x)|\leq\rho_0(x)(C_0+|ln(x^{N-1}\rho_0(x))|),\  \ \ a.e.,
\end{array}
\end{equation}
for some constant $C_0>0$. 
 Then for any $T>0$, there exists $C=C(T)>0$ such that, there exists a global entropy solution $(\rho(x,t),m(x,t))$ of initial boundary value problem \eqref{eq118}-\eqref{eq114} in the sense of Definition \ref{def1} so that
\begin{equation}\label{eq142}
\begin{array}{ll} 
0\leq\rho(x,t)\leq \frac{C}{x^{N-1}},\quad |m(x,t)|\leq\rho(x,t)(C+|ln(x^{N-1}\rho(x,t))|),\quad \  a.e.
\end{array}
\end{equation}
\end{theorem}

\vskip 0.1in




\vskip 0.1in

 

Recently, by utilizing $L^p$ compensated compactness framework, Chen et al. \cite{CHWY} investigated the global finite-energy solution for Cauchy problem \eqref{eq111} with $\gamma>\frac{6}{5}$, while the case $\gamma\leq\frac{6}{5}$ is still open.

\

The rest of the paper is arranged as follows. Some preliminaries are introduced in section 2. In section 3,  a family of approximate solutions is constructed by the shock capturing scheme of Lax-Friedrichs type, and the uniform $L^\infty$ estimate of the approximate solutions is obtained through the Riemann invariants. In section 4, the $H^{-1}$ compactness of a family of entropy dissipation measures is derived by the compactness embedding technique. In section 5, a convergent subsequence is achieved by the compensated compactness framework, and the limit of this subsequence is shown to be the global entropy solution of the isothermal Euler-Poisson system \eqref{eq111} with spherical symmetry. 

\section{Preliminaries}\label{section2}

\setcounter{section}{2}

\setcounter{equation}{0}

 In this section, we denote
\begin{equation}\label{eq122}
\begin{aligned}
\displaystyle v:=(\varrho, \omega)^T, \quad f(v):=(f_1, f_2)(v)=(\omega, \frac{\omega^2}{\varrho}+\varrho)^T,\quad  \\
\displaystyle g(v):=(g_1, g_2)(v)=(0, \frac{N-1}{x}\varrho-\frac{\varrho}{x^{N-1}}\int_1^x\varrho(s, t)ds)^T.
\end{aligned}
\end{equation}
The scalar functions $\varrho(x, t)$ and $\omega(x, t)$ in \eqref{eq120}-\eqref{eq1114} satisfy
\begin{equation}\label{eq123}
\left\{
\begin{array}{ll}
v_t+ f(v)_x=g(v),\qquad x\geq1,\quad t\geq0,\\
v|_{t=0}=v_0(x)=(\varrho_0(x), \omega_0(x)).
\end{array}
\right.
\end{equation}

We first introduce some basic facts about the Riemann solution of the  Cauchy problem of the Euler system of \eqref{eq123},
 \begin{equation}\label{eq124}
\begin{aligned}
v_t+ f(v)_x=0,
\end{aligned}
\end{equation}
with the initial Riemann data
\begin{equation}\label{eq125}
v|_{t=0}=\left\{
\begin{array}{ll}
v_-=(\varrho_->0, \omega_-),\quad x<x_0,\quad x_0>1,\\
v_+=(\varrho_+>0, \omega_+),\quad  x>x_0,
\end{array}
\right.
\end{equation}
and the Riemann solution of the initial-boundary problem of \eqref{eq124} with the initial-boundary Riemann data
\begin{equation}\label{eq126}
\left\{
\begin{array}{ll}
v|_{t=0}=v_+=(\varrho_+>0, \omega_+),\quad x>1, \\ 
\omega|_{x=1}=0, 
\end{array}
\right.
\end{equation}
where $\varrho_{\pm}$ and $\omega_{\pm}$ are constants with $|\frac{\omega_{\pm}}{\varrho_{\pm}}|<+\infty$. 


The Riemann invariants of \eqref{eq124} are
\begin{equation}\label{eq127}
\begin{aligned}
w=\frac{\omega}{\varrho}+\log \varrho, \qquad z=\frac{\omega}{\varrho}-\log \varrho,
\end{aligned}
\end{equation}
and the eigenvalues are
\begin{equation}\label{eq128}
\begin{aligned}
\lambda_1=\frac{\omega}{\varrho}-1, \qquad \lambda_2=\frac{\omega}{\varrho}+1.
\end{aligned}
\end{equation}

The Rankine-Hugoniot condition for the jump in a weak solution to \eqref{eq124} is
\begin{equation}\label{eq129}
\begin{aligned}
\sigma(v-v_0)=f(v)-f(v_0),
\end{aligned}
\end{equation}
where $\sigma$ is the propagating speed of the discontinuity, $v_0=(\varrho_0, \omega_0)$ and $v=(\varrho,\omega)$ are the corresponding left state and right state respectively.
If a discontinuity satisfies the entropy condition
\begin{equation}\label{eq131}
\begin{aligned}
\sigma(\eta(v)-\eta(v_0))-(q(v)-q(v_0))\geq0,
\end{aligned}
\end{equation}
for any convex entropy pair $(\eta,q)$, we call this discontinuity is a shock.

For the problem \eqref{eq124}$-$\eqref{eq126}, the Riemann solutions generally contain two distinct types of rarefaction waves and shock waves, which labeled $1-$rarefaction or $2-$rarefaction waves and $1-$shock or $2-$shock waves, see \cite{CDL3}.

Here we give the following properties about the Riemann solution.

\vskip 0.1in
\begin{lemma} \label{1hs} (\cite{CDL3}) There exists a global entropy solution $(\varrho, \omega)$ for the Riemann problem \eqref{eq124}, \eqref{eq125}, which is a piecewise smooth function satisfying
\begin{equation}\label{eq132}
\left\{
\begin{array}{ll}
w(x, t)\equiv w(\varrho(x, t)>0,\omega(x,t))\leq max\{w(\varrho_-,\omega_-),w(\varrho_+,\omega_+)\},\\
z(x, t)\equiv z(\varrho(x, t)>0,\omega(x,t))\geq min\{z(\varrho_-,\omega_-),z(\varrho_+,\omega_+)\},
\end{array}
\right.
\end{equation}
and for the Riemann problem \eqref{eq124}, \eqref{eq126}, also there exists a global entropy solution $(\varrho, \omega)$, which is piecewise smooth function satisfying
\begin{equation}\label{eq133}
\left\{
\begin{array}{ll}
w(x, t)\equiv w(\varrho(x, t)>0,\omega(x,t))\leq max\{w(\varrho_+,\omega_+),-z(\varrho_+,\omega_+)\},\\
z(x, t)\equiv z(\varrho(x, t)>0,\omega(x,t))\geq min\{0,z(\varrho_+,\omega_+)\},
\end{array}
\right.
\end{equation}
where $w$ and $z$ are the Riemann invariants in \eqref{eq127}.
\end{lemma}
\vskip 0.1in

If the initial Riemann data lies in the regions $$\Theta=\{(\varrho,\omega):w\leq\max( w_0,z_0), z\geq z_0\},$$ then the corresponding Riemann solution $(\varrho(x,t),\omega(x,t))$ of the Riemann problem \eqref{eq124} with \eqref{eq125} or \eqref{eq126} also lies in $\Theta$ which is called invariant regions. And the integral average lemma can be listed as follows.

\vskip 0.1in

\begin{lemma} \label{2hs} (\cite{DCL1}) If $(\varrho, \omega):(a,b)\rightarrow\Theta$, then
\begin{equation}\label{eq134}
\begin{aligned}
\big(\frac{1}{b-a}\int^b_a\varrho(x)dx, \frac{1}{b-a}\int^b_a\omega(x)dx\big)\in\Theta.
\end{aligned}
\end{equation}
\end{lemma}

\vskip 0.1in

\vskip 0.1in

\begin{lemma} \label{6hs} (\cite{CDL3}) Suppose that $v(x,t)=(\varrho(x,t),\omega(x,t))$ is a Riemann solution with central point $(\frac{l}{2},0)$ on the rectangle: $0<x<l$, $0\leq t<h$. Then
\begin{equation}\label{eq13611}
\begin{aligned}
\displaystyle\int_{0}^l|v(x,t)-v(x,h-0)|^2dx&\leq C\int_{0}^l|\varrho(x,t)-\varrho(x,h-0)|^2dx\\
&\leq Cl\sum|\epsilon(\varrho(x,h-0))|^2,
\end{aligned}
\end{equation}
where $\epsilon(\varrho(x,h-0))$ denotes the jump strength of $\varrho(x,h-0)$ across the elementary wave on $t=h$, $C$ only depends on the upper bound of $\varrho(x,t)$ and $|\frac{\omega(x,t)}{\varrho(x,t)}|$, the summation is taken over all jump strengths in $\varrho(x,h-0)$ across elementary waves.
\end{lemma}

\vskip 0.1in


%


\vskip 0.1in

\begin{lemma} \label{3hs} 
 If $g(x)$ is a piecewise continuous function defined on some interval $[a,a+l]$, it consists of constant left state $g_l$ with length $l_1$, intermediate state $g_m$ with length $l_2$, and right state $g_r$ with length $l_3$, $l_1+l_2+l_3= l$, then
\begin{equation}\label{eq136}
\begin{aligned}
\int^{a+l}_{a}|g(x)-\bar{g}|^2dx= l\big[\frac{l_1}{l}\frac{l_3}{l}(g_r-g_l)^2+\frac{l_2}{l}\frac{l_3}{l}(g_r-g_m)^2+\frac{l_1}{l}\frac{l_2}{l}(g_m-g_l)^2\big],
\end{aligned}
\end{equation}
where $$\bar{g}=\frac{1}{l}\int_a^{a+l}g(x)dx.$$
\end{lemma}

\begin{proof} 
Notice that $\bar{g}=\frac{l_1g_l+l_2g_m+l_3g_r}{l}$, let 
$$g_l=\frac{l_1g_l+l_2g_l+l_3g_l}{l}, \ g_m=\frac{l_1g_m+l_2g_m+l_3g_m}{l}, \ g_r=\frac{l_1g_r+l_2g_r+l_3g_r}{l}.$$
Then we have 
\begin{equation}\label{eq11336}
\begin{aligned}
&\int^{a+l}_{a}|g(x)-\bar{g}|^2dx=(g_l-\bar{g})^2l_1+(g_m-\bar{g})^2l_2+(g_r-\bar{g})^2l_3\\
&=\big(\frac{(g_l-g_m)l_2+(g_l-g_r)l_3}{l}\big)^2l_1+\big(\frac{(g_m-g_l)l_1+(g_m-g_r)l_3}{l}\big)^2l_2\\
&\qquad \quad+\big(\frac{(g_r-g_l)l_1+(g_r-g_m)l_2}{l}\big)^2l_3\\
&=(\frac{l_1l_2^2}{l^2}+\frac{l_2l_1^2}{l^2}+\frac{l_1l_2l_3}{l^2})(g_l-g_m)^2+(\frac{l_1l_3^2}{l^2}+\frac{l_3l_1^2}{l^2}+\frac{l_1l_2l_3}{l^2})(g_l-g_r)^2\\
&\qquad \quad+(\frac{l_2l_3^2}{l^2}+\frac{l_3l_2^2}{l^2}+\frac{l_1l_2l_3}{l^2})(g_m-g_r)^2,
\end{aligned}
\end{equation}
which is equivalent to \eqref{eq136}. Then Lemma \ref{3hs} is proved.
\end{proof}

\vskip 0.1in




%
%

On the other hand, Huang-Wang established a compactness framework of approximate solution for isothermal conditions in \cite{HW}, as follows:

\vskip 0.1in

\begin{theorem}\label{theorem0} Let $(\varrho^l(x,t),\omega^l(x,t))$ be a sequence and satisfy the conditions:
$$(i) \qquad 0\leq\varrho^l(x,t)\leq C,\quad |\omega^l(x,t)|\leq\varrho^l(x,t)(C+|\log\varrho^l(x,t)|),\qquad a.e. \qquad\quad $$ for some constant $C$.

$(ii)$ The sequence of entropy dissipation measures $$\eta_t(\varrho^l,\omega^l)+q_x(\varrho^l,\omega^l)$$ is compact in $H_{loc}^{-1}(\mathbb{R}_+\times\mathbb{R}_+)$ for any 
fixed $\xi\in (-1,1)$ as $l$ varies, where 
\begin{equation}\label{eq139}
\begin{aligned}
\eta=\varrho^{\frac{1}{1-\xi^2}}e^{\frac{\xi}{1-\xi^2}\frac{\omega}{\varrho}}, \quad q=(\frac{\omega}{\varrho}+\xi)\varrho^{\frac{1}{1-\xi^2}}e^{\frac{\xi}{1-\xi^2}\frac{\omega}{\varrho}}=(\frac{\omega}{\varrho}+\xi)\eta.
\end{aligned}
\end{equation}
 Then there exists $(\varrho,\omega)$ and a subsequence (still denoted) $(\varrho^l,\omega^l)$ such that
\begin{equation}\label{eq137}
\begin{array}{ll}
(\varrho^l,\omega^l)\rightarrow(\varrho,\omega)  \quad \mbox{as} \quad l\rightarrow0   \quad  \mbox{a.e.} \quad \mbox{in  } L_{loc}^p(\mathbb{R}_+\times\mathbb{R}_+),\quad p\in [1,+\infty),
\end{array}
\end{equation}
and
\begin{equation}\label{eq138}
\begin{array}{ll}
0\leq\varrho\leq C,\quad |\omega(x,t)|\leq\varrho(x,t)(C+|\log\varrho(x,t)|).
\end{array}
\end{equation}
\end{theorem}
\vskip 0.1in

We will prove the strong convergence of our approximate solutions by the above compactness theorem.

\section{Approximate Solutions}\label{section3}
\setcounter{section}{3}

\setcounter{equation}{0}

\noindent In this section, we use a shock capturing scheme of Lax-Friedrichs type developed in \cite{CL} to construct a sequence of approximate solutions $(\varrho^l,\omega^l)$ for the initial-boundary value problem \eqref{eq120} $-$ \eqref{eq1114}. 

Let $0<l<1$ and
\begin{equation}\label{eq148}
\begin{aligned}
h=\frac{l}{10(1+|\log l|)}
\end{aligned}
\end{equation}
 be the space and time mesh length respectively. We partition $(1,+\infty)$ into cells with the $j$-$th$ cell centred at $x_j=1+jl, j=0,1,2...$ and partition $\mathbb{R}_+$ by the sequence $t_i=ih, i\in Z_+$. Denote $v^l=(\varrho^l,\omega^l)$ as the approximate solutions satisfying the Courant-Friendrichs-Lewy (CFL) stability condition:
 \begin{equation}\label{eq149}
\begin{aligned}
\max\limits_{k=1,2}\sup\limits_{\varrho^l, \omega^l}|\lambda_k(\varrho^l,\omega^l)|<\frac{l}{h}.
\end{aligned}
\end{equation}

To ensure the (CFL) condition hold, 
 we use the cut-off technique to make the approximate density functions stay away from vacuum 
by $l^\beta$, $3\leq\beta\leq4$.


Let $\bar{v}_0=(\bar{\varrho}_0,\bar{\omega}_0),$
\begin{equation}\label{eq153}
\begin{aligned}
 \bar{\varrho}_0=\left\{
\begin{array}{ll}
l^{\beta},\quad x\ge\frac{1}{l},\\
\max\{\varrho_0,l^\beta\},\quad 1\le x<\frac{1}{l}; 
\end{array}
\right.
\qquad
 \bar{\omega}_0=\left\{
\begin{array}{ll}
0, \quad x\ge\frac{1}{l},\\
\omega_0, \quad 1\le x<\frac{1}{l}.
\end{array}
\right.
\end{aligned}
\end{equation}
We define the boundary conditions as 
\begin{equation}\label{eq154}
\underline{v}(x,0+0)=\left\{
\begin{array}{ll}
\displaystyle\bar{v}_j^0=\frac{1}{2l}\int_{x_{j-1}}^{x_{j+1}}\bar{v}_0dx, \quad x_{j-1}\leq x\leq x_{j+1},\ j\geq4 \ \mbox{even} ,\\[3mm]
\displaystyle\bar{v}_2^0=\frac{1}{3l}\int_{1}^{1+3l}\bar{v}_0dx, \quad 1\leq x\leq x_3.
\end{array}
\right.
\end{equation}
Then we solve the Riemann problem \eqref{eq125} in the region
\begin{equation}\label{eq155}
\begin{aligned}
R_j^1=\{(x,t),  x_j\leq x\leq x_{j+2}, 0\leq t<t_1, j\geq2 \ \mbox{integers}\},
\end{aligned}
\end{equation}
with the Riemann data
\begin{equation}\label{eq156}
\underline{v}^l|_{t=0}=\left\{
\begin{array}{ll}\bar{v}_j^0,\quad x<x_{j+1},\\
\bar{v}_{j+2}^0,\quad x>x_{j+1},\quad j=2,4...,
\end{array}
\right.
\end{equation}
and the initial-boundary Riemann problem \eqref{eq126} in
\begin{equation}\label{eq158}
\begin{aligned}
\{(x,t),  1\leq x\leq x_2, 0\leq t<t_1\},
\end{aligned}
\end{equation}
with the Riemann data:
\begin{equation}\label{eq160}
\begin{aligned}
\underline{v}^l|_{t=0}=\bar{v}_2^0,  \quad 1< x,
\quad \omega|_{x=1}=0 ,
\end{aligned}
\end{equation}
to obtain $\underline{v}^l(x,t)$ for $0\leq t<t_1$. Set
\begin{equation}\label{eq157}
\begin{aligned}
v^l(x,t)=\underline{v}^l(x,t)+g(\underline{v}^l(x,t))t,\  \  0<t<t_1,
\end{aligned}
\end{equation}
where $g((\underline{v}^l(x,t))=(g_1,g_2)((\underline{v}^l(x,t))$ was given in \eqref{eq122}.

Suppose that the approximate solutions $v^l(x,t)$ have been well defined for $0\leq t<t_i$, and set
\begin{equation}\label{eq161}
\begin{aligned}
 \bar{\varrho}^l=\ \mbox{max}\{\varrho^l,l^\beta\},\quad \bar{\omega}^l=\omega^l, \quad\ \mbox{for} \ \  t_{i-1}\leq t<t_i.
\end{aligned}
\end{equation}
Define $\underline{v}(x,t_i+0)$: 

$(i)$ when $i\geq1$ is odd, $j\geq3$ also is odd,
\begin{equation}\label{eq162}
\underline{v}(x,t_i+0)=\left\{
\begin{array}{ll}
\displaystyle\bar{v}_j^i=\frac{1}{2l}\int_{x_{j-1}}^{x_{j+1}}\bar{v}^l(x,t_i-0)dx, \quad x_{j-1}\leq x\leq x_{j+1},\\[3mm]
\displaystyle\bar{v}_1^i=\frac{1}{2l}\int_{1}^{1+2l}\bar{v}^l(x,t_i-0)dx, \quad 1\leq x\leq x_2;
\end{array}
\right.
\end{equation}

$(ii)$ when $i\geq2$ is even, $j\geq4$ also is even,
\begin{equation}\label{eq163}
\underline{v}(x,t_i+0)=\left\{
\begin{array}{ll}
\displaystyle\bar{v}_j^i=\frac{1}{2l}\int_{x_{j-1}}^{x_{j+1}}\bar{v}^l(x,t_i-0)dx, \quad x_{j-1}\leq x\leq x_{j+1},\\[3mm]
\displaystyle\bar{v}_2^i=\frac{1}{3l}\int_{1}^{1+3l}\bar{v}^l(x,t_i-0)dx, \quad 1\leq x\leq x_3.
\end{array}
\right.
\end{equation}
Then we solve the Riemann problem in the strip region
\begin{equation}\label{eq164}
\begin{aligned}
R_j^i=\{(x,t),  x_j\leq x\leq x_{j+2}, t_i\leq t<t_{i+1}\},
\end{aligned}
\end{equation}
with the Riemann data
\begin{equation}\label{eq165}
\underline{v}^l|_{t=t_i}=\left\{
\begin{array}{ll}\bar{v}_j^i,\quad x<x_{j+1},\\
\bar{v}_{j+2}^i,\quad x>x_{j+1},
\end{array}
\right.
\end{equation}
for $j\geq1$ odd when $i\geq1$ is odd and for $j\geq2$ even when $i\geq2$ is even,
and the initial-boundary Riemann problem \eqref{eq126} in
\begin{equation}\label{eq166}
\begin{aligned}
\{(x,t),  1\leq x\leq x_1, t_i\leq t<t_{i+1}\},
\end{aligned}
\end{equation}
when $i$ is odd with the Riemann data:
\begin{equation}\label{eq169}
\begin{aligned}
\underline{v}^l|_{t=t_i}=\bar{v}_1^i,  1< x\leq x_1,\quad \omega|_{x=1}=0 ;
\end{aligned}
\end{equation}
and in
\begin{equation}\label{eq168}
\begin{aligned}
\{(x,t),  1\leq x\leq x_2, t_i\leq t<t_{i+1}\},
\end{aligned}
\end{equation}
when $i$ is even with the Riemann data:
\begin{equation}\label{eq167}
\begin{aligned}
\underline{v}^l|_{t=t_i}=\bar{v}_2^i,  1< x\leq x_2,\quad \omega|_{x=1}=0 ,
\end{aligned}
\end{equation}
to obtain $\underline{v}^l(x,t)$ for $t_i\leq t<t_{i+1}$. Set
\begin{equation}\label{eq170}
\begin{aligned}
v^l(x,t)=\underline{v}^l(x,t)+g(\underline{v}^l(x,t))(t-t_i),
\end{aligned}
\end{equation}
for $t_i<t<t_{i+1}$.
Let $\bar{v}^l(x,t)=\big(\bar{\varrho}^l(x,t),\bar{\omega}^l(x,t)\big)$ be the approximate solutions, where
\begin{equation}\label{eq171}
\begin{aligned}
 \bar{\varrho}^l(x,t)=\ \mbox{max}\{\varrho^l(x,t),l^\beta\},\quad \bar{\omega}^l(x,t)=\omega^l(x,t).
\end{aligned}
\end{equation}

Denote
\begin{equation}\label{eq172}
\begin{array}{ll}
w^l(x,t)=w\big(v^l(x,t)\big),\quad \underline{w}^l(x,t)=w\big(\underline{v}^l(x,t)\big),\quad \bar{w}^l(x,t)=w\big(\bar{v}^l(x,t)\big),\\
z^l(x,t)=z\big(v^l(x,t)\big),\quad \underline{z}^l(x,t)=z\big(\underline{v}^l(x,t)\big),\quad \bar{z}^l(x,t)=z\big(\bar{v}^l(x,t)\big).
\end{array}
\end{equation}

We have
\vskip 0.1in

\begin{lemma}\label{31hs} Suppose that the initial data $(\varrho_0(x),\omega_0(x))$ satisfies
\begin{equation}\label{eq173}
\begin{aligned}
0\leq\varrho_0\leq C_0,\quad |\omega_0(x)|\leq\varrho_0(x)(C_0+|\log\varrho_0|),\ \ \ a.e.
\end{aligned}
\end{equation}
where $C_0>0$ is a constant. Then there exists some constant $l_0>0$ such that if $l\leq l_0$, 
there exists constant $C(T,M)>0$ 
such that
\begin{equation}\label{eq174}
\begin{aligned}
l^\beta\leq\bar{\varrho}^l(x,t)\leq C(T,M),\quad |\bar{\omega}^l(x,t)|\leq\bar{\varrho}^l(x,t)(C(T,M)+|\log\bar{\varrho}^l(x,t)|),
\end{aligned}
\end{equation}
hold for any $T>0$.
\end{lemma}
\vskip 0.1in

\begin{proof}

    First of all, 
    for the initial-boundary Riemann solution $\underline \varrho^l(x,t)$, we claim that       
\begin{equation}\label{eq176a}
\begin{aligned}
\int_1^{x_2}\underline{\varrho}^l(x,h-0)dx\le\int_1^{x_3 }\underline{\varrho}^l(x,0)dx.
\end{aligned}
\end{equation}    
Indeed, take four points $$A:(1,h),\ B:(x_2,h),\ C:(1,0), \ D:(x_3, 0)$$ to form a trapezoid in the $(x,t)$ plane, the lines $AD$ and $BC$ are $$AD:\ x(t)=1;\ \ BC:\ x(t)=-\frac{l}{h}t+x_3, \ \ t\in[0,h).$$
 Using the Divergent Theorem, R-H condition and CFL condition gives \eqref{eq176a}.  
Similarly, for $(i-1)h\le t<ih, $ $i+j=$ odd, 
 \begin{equation}\label{eq178a}
\begin{aligned}
\int_1^{x_j}\underline{\varrho}^l(x,ih-0)dx\le\int_1^{-\frac{l}{h}(t-(i-1)h)+x_{j+1}}\underline{\varrho}^l(x,t)dx\le\int_1^{x_{j+1} }\underline{\varrho}^l(x,(i-1)h+0)dx.
\end{aligned}
\end{equation}  
Now, we show that $\underline \varrho^l$ are uniformly bounded with respect to $l$. For $(i-1)h\le t<ih, $ $1\le x<\frac{1}{l}+il$, we derive   
\begin{align}\label{eq179a}
\begin{aligned}
\int_1^{\frac{1}{l}}\underline{\varrho}^l(x,t)dx&\le\int_1^{-\frac{l}{h}(t-ih)+\frac{1}{l}}\underline{\varrho}^l(x,t)dx\le\int_1^{\frac{1}{l}+l} \underline{\varrho}^l(x,(i-1)h+0)dx\\
&\le\int_1^{\frac{1}{l}+l} \underline{\varrho}^l(x,(i-1)h-0)dx+[\frac{1}{l}+l-1]l^\beta\\
&\le\int_1^{\frac{1}{l}+il} {\varrho}_0(x)dx+[\frac{i}{l}+l+2l+\cdots+il-i]l^\beta\\
&\le\int_1^{\frac{1}{l}+il} {\varrho}_0(x)dx+[\frac{i}{l}+\frac{(i+1)i}{2}l-i]l^\beta.
\end{aligned}
\end{align}     
Notice that \eqref{eq148} and $3\le\beta\le4$, there exists $l_0$ such that when $l\le l_0$,   
 \begin{align}\label{eq180a}
\begin{aligned}
\big[\frac{i}{l}+\frac{(i+1)i}{2}l-i\big]l^\beta=il^{\beta-1}+\frac{(i+1)i}{2}l^{\beta+1}-il^\beta< Tl^{\beta-\frac{5}{2}},
\end{aligned}
\end{align}   
for $i=\frac{T}{h}=10T(1+|\log l|)l^{-1}$.  
Thus, we obtain   
 \begin{align}\label{eq181a}
\begin{aligned}
\int_1^{+\infty}\underline{\varrho}^l(x,t)dx\le C,
\end{aligned}
\end{align}     
 for any $t\le T$, $x\in[1,+\infty)$.

From the initial data \eqref{eq173}, there exists some reasonably large constant $\alpha_0>0$ such that
\begin{equation}\label{eq180}
\begin{aligned}
\sup w\big(\varrho_0(x),\omega_0(x)\big)\leq\alpha_0, \qquad \inf z\big(\varrho_0(x),\omega_0(x)\big)\geq-\alpha_0.
\end{aligned}
\end{equation}
And we get the Riemann invariants corresponding to the cut-off function from \eqref{eq154}
\begin{equation}\label{eq181}
\begin{array}{ll}
\displaystyle\underline{w}_0(x)=w(\underline{\varrho}_0(x),\underline{\omega}_0(x))=\frac{\underline{\omega}_0(x)}{\underline{\varrho}_0(x)}+\log\underline{\varrho}_0(x)\leq\alpha_0,\\[3mm]
\displaystyle\underline{z}_0(x)=z(\underline{\varrho}_0(x),\underline{\omega}_0(x))=\frac{\underline{\omega}_0(x)}{\underline{\varrho}_0(x)}-\log\underline{\varrho}_0(x)\geq-\alpha_0.
\end{array}
\end{equation}

According to the properties of Riemann invariants in Lemma \ref{1hs} and \ref{2hs}, the Riemann solutions $(\underline{\varrho}^l,\underline{\omega}^l)(x,t)$ satisfy
\begin{equation}\label{eq182}
\begin{array}{ll}
\underline{w}^l(x,t)\leq\alpha_0,\qquad \underline{z}^l(x,t)\geq-\alpha_0,
\end{array}
\end{equation}
 for $0<t<t_1$.
From \eqref{eq170}, \eqref{eq181a} and Lemma \ref{1hs}, we give the following estimate
\begin{equation}\label{eq183}
\begin{array}{ll}
\displaystyle w^l(x,t)&\displaystyle=\underline{w}^l(x,t)+(\frac{N-1}{x}-\frac{1}{x^{N-1}}\int_1^x\underline{\varrho}^l(y,t)dy)(t-t_i)\\
&\displaystyle\leq\sup\underline{w}^l(x,t)+(N-1)h\\
&\leq\alpha_0+(N-1)h,\\
z^l(x,t)&\displaystyle=\underline{z}^l(x,t)+(\frac{N-1}{x}-\frac{1}{x^{N-1}}\int_1^x\underline{\varrho}^l(y,t)dy)(t-t_i)\\
&\displaystyle\geq\inf\underline{z}^l(x,t)-C(T,M)h\\
&\displaystyle\geq-\alpha_0-C(T,M)h,
\end{array}
\end{equation}
for $t_i\leq t\leq t_{i+1}$.
Then, we get
\begin{equation}\label{eq184}
\begin{aligned}
w^l(x,t)=w(\varrho^l,\omega^l)\leq\alpha_0+Ch, \qquad z^l(x,t)=z(\varrho^l,\omega^l)\geq-\alpha_0-Ch,
\end{aligned}
\end{equation}
where constant $C$ only depends on $T$ and $M$.
Repeating the same procedure, we have
\begin{equation}\label{eq185}
\begin{aligned}
\bar{w}^l(x,t)=w(\bar{\varrho}^l,\bar{\omega}^l)\leq\alpha_0+CT, \quad \bar{z}^l(x,t)=z(\bar{\varrho}^l,\bar{\omega}^l)\geq-\alpha_0-CT,\quad t\in[0,T)
\end{aligned}
\end{equation}
for any $T>0$.
Thus, from \eqref{eq127}, there exists $C(T,M)>0$ such that
 \begin{equation}\label{eq186}
 \begin{aligned}
  0< l^\beta\leq\bar{\varrho}^l(x,t)\leq C(T,M),\qquad\qquad\qquad \qquad\quad\\
\bar{\varrho}^l(x,t)(C(T,M)+\log\bar{\varrho}^l(x,t))\leq\bar{\omega}^l(x,t)\leq \bar{\varrho}^l(x,t)(C(T,M)-\log\bar{\varrho}^l(x,t)),
\end{aligned}
\end{equation}
for any $ 0\leq t\leq T$.
\end{proof}

From Lemma \ref{31hs}, we get the uniform bound for the approximate solutions in $\bar{\Pi}_T$, there exists $l_0=l_0(T)>0$ such that the (CFL) stability condition holds for any $0<l\leq l_0$, which means that the approximate solutions are well-defined in $\Pi_T$.

\section{$H^{-1}$ compactness of entropy dissipation measures}\label{section4}

\setcounter{section}{4}

\setcounter{equation}{0}

\noindent In this section, we shall prove the $H^{-1}$ compactness of entropy dissipation measures
$$\partial_t\eta(\bar{v}^l)+\partial_xq(\bar{v}^l)$$
associated with weak entropy pair $(\eta,q)$ in \eqref{eq139} for the approximate solution $\bar{v}^l$. For this purpose, we use the following lemma, whose proof can be found in \cite{CL, HLY}.

\vskip 0.1in

\begin{lemma}\label{41hs}  Let $\Omega\subset\mathbb{R}^N$ be a bounded open set. Then there exists
$$(\mbox{Compact set of }W^{-1,p}(\Omega))\cap(\mbox{Bounded set of }W^{-1,r}(\Omega))\subset(\mbox{Compact set of }H^{-1}_{loc}(\Omega))$$
for some constants $p$ and $r$ that satisfy $1<p\leq2<r<\infty$.
\end{lemma}
\vskip 0.1in

By the Lemma \ref{41hs}, we apply the duality and the Sobolev interpolation inequality to obtain the $H^{-1}$ compactness. For simplicity, we drop the superscript $l$ of the approximate solution $\bar{v}^l$.

\vskip 0.1in

\begin{theorem}\label{theorem41}  Assume that the initial data satisfies \eqref{eq173}, then the sequence of the entropy dissipation measure 
 \begin{equation}\label{eq187}
\begin{aligned}
\partial_t\eta(\bar{v})+\partial_xq(\bar{v}) \  \mbox{is compact in }H^{-1}_{loc}(\Pi),
\end{aligned}
\end{equation}
for any weak entropy pair $(\eta,q)$ in \eqref{eq139} with $\xi\in(-1,1)$.
\end{theorem}
\vskip 0.1in

\begin{proof}
    
    To present the proof process clearly, we divide it into the following three steps.

{\bf step 1. } 
For the approximate solution $\bar{v}$, we claim that
 \begin{equation}\label{eq188}
\begin{aligned}
\sum\limits_{ih\leq T}\sum\limits_{1+(j+1)l\leq L}\int_{x_{j-1}}^{x_{j+1}}(\bar{v}(x,t_i-0)-\bar{v}^i_j)^2dx\leq C(T,M),
\end{aligned}
\end{equation}
holds for any $L>1$, where $C(T,M)$ is some constant and $\bar{v}^i_j$ defined in \eqref{eq154}, \eqref{eq162} and \eqref{eq163}.

Indeed, a strictly convex entropy $\eta$ was fixed in \eqref{eq139} with $\xi\in(-\frac{1}{2},\frac{1}{2})$. In the time strip $0\leq t_i\leq t\leq t_{i+1}\leq T$, 
 by the Green formula for the Riemann solutions $\underline{v}$, we have
\begin{equation}\label{eq189}
\begin{aligned}
\sum\limits_{1+(j+1)l\leq L}\int_{x_{j-1}}^{x_{j+1}}\eta(\underline{v}^{\underline{i+1}})-\eta(\bar{v}_j^i)dx+\int_{t_i}^{t_{i+1}}\sum(\sigma[\eta]-[q])dt\leq0,
\end{aligned}
\end{equation}
where $\underline{v}^{\underline{i+1}}=\underline{v}(x,t_{i+1}-0)$
, the summation $\sum$ is taken over all shock waves in $\underline{v}$ at fixed time $t\in[t_i,t_{i+1}]$. $\sigma$ is the propagating speed of the shock wave, $[\eta]$ and $[q]$ denote the jump of $\eta(\underline{v}(x,t))$ and $q(\underline{v}(x,t))$ across the shock wave in $\underline{v}(x,t)$ from the left to the right, i.e.
\begin{equation}\label{eq190}
\begin{aligned}
\ [\eta]=\eta(\underline{v}(x(t)+0,t))-\eta(\underline{v}(x(t)-0,t)), \ [q]=q(\underline{v}(x(t)+0, t))-q(\underline{v}(x(t)-0,t)),
\end{aligned}
\end{equation}
where $x(t)$ is the location that the shock wave occurs. Summing over all $i$ in \eqref{eq189} for $t\in (0,T)$, we have
\begin{equation}\label{eq192}
\begin{aligned}
\displaystyle\sum\limits_{i,j}\int_{x_{j-1}}^{x_{j+1}}&\eta(\underline{v}^{\underline{i}})-\eta(v^{\underline{i}})dx+\sum\limits_{i,j}\int_{x_{j-1}}^{x_{j+1}}\eta(v^{\underline{i}})-\eta(\bar{v}^{\underline{i}})dx\\
\displaystyle &+\sum\limits_{i,j}\int_{x_{j-1}}^{x_{j+1}}\eta(\bar{v}^{\underline{i}})-\eta(\bar{v}_j^i)dx +\int_{0}^{T}\sum(\sigma[\eta]-[q])dt\\
&\qquad=I_1+I_2+I_3+I_4\leq C(T,M),
\end{aligned}
\end{equation}
with $1\leq1+jl\leq L$, $0\leq ih\leq T$, $i+j$ even, and $v^{\underline{i}}=v(x,t_i-0)$.

First, the entropy inequality
\begin{equation}\label{eq191}
\begin{aligned}
I_4=\int_{0}^{T}\sum(\sigma[\eta]-[q])dt\geq0
\end{aligned}
\end{equation}
holds because of the convexity for $\eta$ in \eqref{eq139}.

Noting that
\begin{equation}\label{eq194}
\begin{array}{ll}
\nabla\eta=(\partial_\varrho\eta,\partial_\omega\eta)=(\frac{1}{(1-\xi^2)\varrho}(1-\xi\frac{\omega}{\varrho})\eta,\frac{\xi}{(1-\xi^2)\varrho}\eta),
\end{array}
\end{equation}
and
\begin{equation}\label{eq195}
\begin{array}{ll}
\displaystyle\nabla^2\eta=\left|
 \begin{array}{ccc}
\displaystyle \eta_{\varrho\varrho} & \displaystyle \eta_{\varrho\omega}\\
\displaystyle \eta_{\omega\varrho} & \displaystyle\eta_{\omega\omega} \\
  \end{array}
  \right| =\left|
 \begin{array}{ccc}
 \displaystyle\frac{\xi^2}{(1-\xi^2)^2\varrho^2}\eta(\frac{\omega}{\varrho^2}-\frac{2\xi\omega}{\varrho}+1) & \displaystyle \frac{\xi^2}{(1-\xi^2)^2\varrho^2}\eta(\xi-\frac{\omega}{\varrho})\\[3mm]
\displaystyle \frac{\xi^2}{(1-\xi^2)^2\varrho^2}\eta(\xi-\frac{\omega}{\varrho}) &\displaystyle \frac{\xi^2}{(1-\xi^2)^2\varrho^2}\eta \\
  \end{array}
  \right|\\
\displaystyle \qquad\qquad =\frac{\xi^4}{(1-\xi^2)^3}\varrho^{\frac{2\xi^2}{1-\xi^2}-2}e^{\frac{2\xi}{1-\xi^2}\frac{\omega}{\varrho}}.
\end{array}
\end{equation}
By the construction step \eqref{eq170}, we have
\begin{equation}\label{eq193}
\begin{array}{ll}
\displaystyle v^{\underline{i}}(x,t)-\underline{v}^{\underline{i}}(x,t)=g(\underline{v}^{\underline{i}}(x,t))h=(0, \frac{N-1}{x}\underline{\varrho}^{\underline{i}}-\frac{\underline{\varrho}^{\underline{i}}}{x^{N-1}}\int_0^x\underline{\varrho}^{\underline{i}}(y,t)dy)h,
\end{array}
\end{equation}
with $$\varrho^{\underline{i}}=\underline{\varrho}^{\underline{i}}, \quad \omega^{\underline{i}}=\underline{\omega}^{\underline{i}}+\frac{N-1}{x}\underline{\varrho}^{\underline{i}}h-\frac{\underline{\varrho}^{\underline{i}}}{x^{N-1}}h\int_0^x\underline{\varrho}^{\underline{i}}(y,t)dy.$$

Let $\omega_\theta^i=\underline{\omega}^{\underline{i}}+\theta(\omega^{\underline{i}}-\underline{\omega}^{\underline{i}}),\ 0\leq\theta\leq1$. 
We have
\begin{equation}\label{eq196}
\begin{array}{ll}
\displaystyle |I_1|=|\sum\limits_{i,j}\int_{x_{j-1}}^{x_{j+1}}\eta(\underline{v}^{\underline{i}})-\eta(v^{\underline{i}})dx|\\
\displaystyle\leq\sum\limits_{i,j}\int_{x_{j-1}}^{x_{j+1}}|[\int_0^1\nabla\eta(\underline{v}^{\underline{i}}+\theta(v^{\underline{i}}-\underline{v}^{\underline{i}}))d\theta](v^{\underline{i}}-\underline{v}^{\underline{i}})|dx\\
\displaystyle=\sum\limits_{i,j}\int_{x_{j-1}}^{x_{j+1}}|\int_0^1\eta_\omega(\varrho^{\underline{i}},\omega_\theta^i)(\omega^{\underline{i}}-\underline{\omega}^{\underline{i}})d\theta|dx\\
\displaystyle=\sum\limits_{i,j}\int_{x_{j-1}}^{x_{j+1}}|\int_0^1\frac{\xi}{(1-\xi^2)}\eta(\varrho^{\underline{i}},\omega_\theta^i)(\frac{N-1}{x}h-\frac{h}{x^{N-1}}\int_1^x\underline{\varrho}^{\underline{i}}(y,t)dy)d\theta|dx\\
\displaystyle\leq\sum\limits_{i,j}\int_{x_{j-1}}^{x_{j+1}}\int_0^1\frac{|\xi|}{(1-\xi^2)}\eta(\varrho^{\underline{i}},\omega_\theta^i)d\theta \frac{N-1}{x}h dx.
\end{array}
\end{equation}
From Lemma \ref{31hs}, there is
\begin{equation}\label{eq197}
\begin{aligned}
\displaystyle-C(T,M)+\log\varrho^{\underline{i}}\leq\frac{\omega^i_\theta}{\varrho^{\underline{i}}}=\frac{\underline{\omega}^{\underline{i}}}{\underline{\varrho}^{\underline{i}}}+\theta(\frac{\omega^{\underline{i}}}{\varrho^{\underline{i}}}-\frac{\underline{\omega}^{\underline{i}}}{\underline{\varrho}^{\underline{i}}})\leq C(T,M)-\log\varrho^{\underline{i}},
\end{aligned}
\end{equation}
where $C(T,M)$ only depends on $T$ and $M$. 
And then
\begin{equation}\label{eq198}
\begin{aligned}
\displaystyle\eta(\varrho^{\underline{i}},\omega_\theta^i)=(\varrho^{\underline{i}})^{\frac{1}{1-\xi^2}}e^{\frac{\xi}{1-\xi^2}\frac{\omega^i_\theta}{\varrho^{\underline{i}}}}\leq C'(\varrho^{\underline{i}})^{\frac{1}{1-\xi^2}}e^{-\frac{|\xi|}{1-\xi^2}\log\varrho^{\underline{i}}}\leq C'(\varrho^{\underline{i}})^{\frac{1}{1-\xi^2}-\frac{|\xi|}{1-\xi^2}}\leq C,
\end{aligned}
\end{equation}
where $\frac{2}{3}\leq\frac{1}{1-\xi^2}\leq1$ with $\xi\in(-\frac{1}{2},\frac{1}{2})$, and $C'$, $C$ are some positive constants.
Thus, we obtain
\begin{equation}\label{eq199}
\begin{aligned}
|I_1|\leq\sum\limits_{i,j}\int_{x_{j-1}}^{x_{j+1}}Chdx\leq C(T,M).
\end{aligned}
\end{equation}
For $I_2$, we have
\begin{equation}\label{eq200}
\begin{array}{ll}
\displaystyle|I_2|=|\sum\limits_{i,j}\int_{x_{j-1}}^{x_{j+1}}\eta(v^{\underline{i}})-\eta(\bar{v}^{\underline{i}})dx|\\
\displaystyle\leq|\sum\limits_{i,j}\int_{x_{j-1}}^{x_{j+1}}\eta(v^{\underline{i}})-\eta(\bar{v}^{\underline{i}})dx|_{\varrho^{\underline{i}}>l^\beta}+|\sum\limits_{i,j}\int_{x_{j-1}}^{x_{j+1}}\eta(v^{\underline{i}})-\eta(\bar{v}^{\underline{i}})dx|_{\varrho^{\underline{i}}\leq l^\beta}\\
\displaystyle\leq0+|\sum\limits_{i,j}\int_{x_{j-1}}^{x_{j+1}}\eta(v^{\underline{i}})-\eta(\bar{v}^{\underline{i}})dx|_{\varrho^{\underline{i}}\leq l^\beta}\\
\displaystyle\leq C(T,L)l^{\beta\frac{1-|\xi
|}{1-\xi^2}}\frac{1}{h}\leq C(T,L)l^{\beta\frac{1-|\xi
|}{1-\xi^2}-1}(1+|\log l|),
\end{array}
\end{equation}
where
\begin{equation}\label{eq2001}
\begin{aligned}
1<\frac{3}{2}\leq\beta\frac{1-|\xi|}{1-\xi^2}\leq\frac{16}{3},
\end{aligned}
\end{equation}
for $\xi\in(-\frac{1}{2},\frac{1}{2})$ and $3\leq\beta\leq4$. Thus, we know
\begin{equation}\label{eq201}
\begin{aligned}
|I_2|\rightarrow 0  \qquad \  \mbox{as} \qquad l\rightarrow 0.
\end{aligned}
\end{equation}
Considering the convexity of $\eta$, we have
\begin{equation}\label{eq202}
\begin{array}{ll}
\displaystyle I_3=\sum\limits_{i,j}\int_{x_{j-1}}^{x_{j+1}}\eta(\bar{v}^{\underline{i}})-\eta(\bar{v}_j^i)dx\\
\displaystyle\quad=\sum\limits_{i,j}\int_{x_{j-1}}^{x_{j+1}}\nabla\eta(\bar{v}^i_j)(\bar{v}^{\underline{i}}-\bar{v}^i_j)dx+\sum\limits_{i,j}\int_{x_{j-1}}^{x_{j+1}}(\bar{v}^{\underline{i}}-\bar{v}^i_j)^\tau\nabla^2\eta(\bar{v}_\theta)(\bar{v}^{\underline{i}}-\bar{v}^i_j)dx\\
\displaystyle\quad \geq0+\sum\limits_{i,j}\int_{x_{j-1}}^{x_{j+1}}\frac{\xi^4}{(1-\xi^2)^3}\bar{\varrho}_\theta^{\frac{2\xi^2}{1-\xi^2}-2}e^{\frac{2\xi}{1-\xi^2}\frac{\bar{\omega}_\theta}{\bar{\varrho}_\theta}}(\bar{v}^{\underline{i}}-\bar{v}^i_j)^2dx\\
\displaystyle\quad \geq C(T,M)\sum\limits_{i,j}\int_{x_{j-1}}^{x_{j+1}}(\bar{v}^{\underline{i}}-\bar{v}^i_j)^2dx\geq0,
\end{array}
\end{equation}
where $\displaystyle\nabla^2\eta(\bar{v}_\theta)=\int_0^1(1-\theta)\nabla^2\eta(\bar{v}^i_j+\theta(\bar{v}^{\underline{i}}-\bar{v}^i_j))d\theta$ and $\bar{v}_\theta=(\bar{\varrho}_\theta,\bar{\omega}_\theta)=\bar{v}^i_j+\theta(\bar{v}^{\underline{i}}-\bar{v}^i_j)$.
And from \eqref{eq192} and \eqref{eq191}, we have
\begin{equation}\label{eq203}
\begin{aligned}
0\leq I_3\leq C(T,M), \qquad 0\leq I_4\leq C(T,M).
\end{aligned}
\end{equation}
Then we conclude \eqref{eq188} directly from \eqref{eq202}.

{\bf step 2. } For any function $\varphi\in C_0^\infty(\Pi_T)$ with $supp\varphi\in\Omega\subset[1,L)\times[0,T]$ where $\Omega$ is an open set, the entropy equality can be written in the form
\begin{equation}\label{eq204}
\begin{aligned}
\iint\limits_{\Pi_T}(\eta(\bar{v})\varphi_t+q(\bar{v})\varphi_x)dxdt=A(\varphi)+B(\varphi)+D(\varphi)+S(\varphi),
\end{aligned}
\end{equation}
where
\begin{equation}\label{eq205}
\begin{aligned}
\displaystyle &A(\varphi)=\sum\limits_{i,j}\int_{x_{j-1}}^{x_{j+1}}(\eta(\underline{v}^{\underline{i}})-\eta(v^{\underline{i}}))\varphi(x,t_i)dx+\sum\limits_{i,j}\int_{x_{j-1}}^{x_{j+1}}(\eta(v^{\underline{i}})-\eta(\bar{v}^{\underline{i}}))\varphi(x,t_i)dx,\\
\displaystyle &B(\varphi)=\iint\limits_{\Pi_T}(\eta(\bar{v})-\eta(\underline{v}))\varphi_t+(q(\bar{v})-q(\underline{v}))\varphi_x)dxdt,\\
\displaystyle &D(\varphi)=\sum\limits_{i,j}\int_{x_{j-1}}^{x_{j+1}}(\eta(\bar{v}^{\underline{i}})-\eta(\bar{v}_j^i))\varphi(x,t_i)dx,\\
\displaystyle &S(\varphi)=\int_{0}^{T}\sum(\sigma[\eta]-[q])\varphi(x(t),t)dt.
\end{aligned}
\end{equation}
From \eqref{eq199} and \eqref{eq201}, we have
\begin{equation}\label{eq206}
\begin{aligned}
|A(\varphi)|\leq C\|\varphi\|_\infty.
\end{aligned}
\end{equation}
And using \eqref{eq203}, there yields
\begin{equation}\label{eq207}
\begin{aligned}
|S(\varphi)|\leq C\|\varphi\|_\infty.
\end{aligned}
\end{equation}
We denote
\begin{equation}\label{eq208}
\begin{array}{ll}
\displaystyle B(\varphi)=\iint\limits_{\Pi_T}(\eta(v)-\eta(\underline{v}))\varphi_t+(q(v)-q(\underline{v}))\varphi_x)dxdt\\
\qquad\qquad+\iint\limits_{\Pi_T}(\eta(\bar{v})-\eta(v))\varphi_t+(q(\bar{v})-q(v))\varphi_x)dxdt\\
\qquad :=B_1(\varphi)+B_2(\varphi).
\end{array}
\end{equation}
Noting that
\begin{equation}\label{eq2007}
\begin{aligned}
&\partial_\varrho q=q_\varrho=-\frac{\omega}{\varrho^2}\eta+(\frac{\omega}{\varrho}+\xi)\eta_\varrho=\big[(\frac{\omega}{\varrho}+\xi)\frac{1-\xi\frac{\omega}{\varrho}}{(1-\xi^2)\varrho}-\frac{\omega}{\varrho^2}\big]\eta,\\
&\partial_\omega q=q_\omega=\frac{1}{\varrho}\eta+(\frac{\omega}{\varrho}+\xi)\eta_\omega=\big[\frac{1}{\varrho}+(\frac{\omega}{\varrho}+\xi)\frac{\xi}{(1-\xi^2)\varrho}\big]\eta.
\end{aligned}
\end{equation}

Set $\omega_\theta=\underline{\omega}+\theta(\omega-\underline{\omega})$.
From \eqref{eq198}, we have
\begin{equation}\label{eq210}
\displaystyle\eta(\varrho,\omega_\theta)|\log\varrho|\leq C'\varrho^{\frac{1-|\xi|}{1-\xi^2}}|\log\varrho|\leq C.
\end{equation}
Then, there exists
\begin{equation}\label{eq209}
\begin{aligned}
\displaystyle &|B_1(\varphi)|\leq|\iint\limits_{\Pi_T}\int_0^1(\eta_\omega(\underline{v}+\theta(v-\underline{v}))d\theta(\omega-\underline{\omega})\varphi_tdxdt|\\
&\qquad+|\iint\limits_{\Pi_T}\int_0^1(q_\omega(\underline{v}+\theta(v-\underline{v}))d\theta(\omega-\underline{\omega})\varphi_xdxdt|\\
\displaystyle &\leq Ch|\iint\limits_{\Pi_T}\int_0^1\eta(\varrho,\omega_\theta)d\theta(\frac{N-1}{x}-\frac{1}{x^{N-1}}\int_0^x\underline{\varrho}(y,t)dy)\varphi_tdxdt|\\
\displaystyle   &\quad+Ch|\iint\limits_{\Pi_T}\int_0^1\eta(\varrho,\omega_\theta)(1+\frac{\omega_\theta}{\varrho}+\xi)d\theta(\frac{N-1}{x}-\frac{1}{x^{N-1}}\int_0^x\underline{\varrho}(y,t)dy)\varphi_xdxdt|\\
\displaystyle & \leq C(T,M)h\|\varphi\|_{H^1_0(\Omega)},
\end{aligned}
\end{equation}
and
\begin{equation}\label{eq211}
\begin{aligned}
\displaystyle & |B_2(\varphi)|\leq|\iint\limits_{\Pi_T}(\eta(\bar{v})-\eta(v))\varphi_tdxdt|+|\iint\limits_{\Pi_T}(q(\bar{v})-q(v))\varphi_x)dxdt|\\
\displaystyle &\leq Cl^{\beta{\frac{1-|\xi|}{1-\xi^2}}}|\iint\limits_{\Pi_T}\varphi_tdxdt|_{\varrho\leq l^\beta}+0|_{\varrho\geq l^\beta}+C|\iint\limits_{\Pi_T}\varrho^{{\frac{1-|\xi|}{1-\xi^2}}}\varphi_x\log\varrho dxdt|_{\varrho\leq l^\beta}+0|_{\varrho\geq l^\beta}\\
\displaystyle &\leq Cl\|\varphi\|_{H^1_0(\Omega)}.
\end{aligned}
\end{equation}
That means
\begin{equation}\label{eq1208}
\begin{array}{ll}
|B(\varphi)|\leq Cl^{\frac{1}{2}}\|\varphi\|_{H^1_0(\Omega)}\leq Cl^{\frac{1}{2}}.
\end{array}
\end{equation}

Now, we set
\begin{equation}\label{eq212}
\begin{aligned}
\displaystyle D(\varphi) &=\sum\limits_{i,j}\int_{x_{j-1}}^{x_{j+1}}(\eta(\bar{v}^{\underline{i}})-\eta(\bar{v}_j^i))\varphi_j^idx+\sum\limits_{i,j}\int_{x_{j-1}}^{x_{j+1}}(\eta(\bar{v}^{\underline{i}})-\eta(\bar{v}_j^i))(\varphi^i-\varphi_j^i)dx\\
\displaystyle &:=D_1(\varphi)+D_2(\varphi),
\end{aligned}
\end{equation}
where $\varphi_j^i=\varphi(x_j,t_i)$ and $\varphi^i=\varphi(x,t_i)$.
By \eqref{eq203}, we get
\begin{equation}\label{eq213}
\begin{aligned}
D_1(\varphi)\leq\|\varphi\|_\infty\sum\limits_{i,j}\int_{x_{j-1}}^{x_{j+1}}(\eta(\bar{v}^{\underline{i}})-\eta(\bar{v}_j^i))dx\leq C\|\varphi\|_\infty,
\end{aligned}
\end{equation}
and
\begin{equation}\label{eq214}
\begin{aligned}
\displaystyle &|D_2(\varphi)|\leq\sum\limits_{i,j}\int_{x_{j-1}}^{x_{j+1}}|(\eta(\bar{v}^{\underline{i}})-\eta(\bar{v}_j^i))(\varphi^i-\varphi_j^i)|dx\\
\displaystyle & \leq C\|\varphi\|_{C^\alpha}l^\alpha\sum\limits_{i,j}\int_{x_{j-1}}^{x_{j+1}}|\int_0^1\nabla\eta(\bar{v}_\theta)(\bar{v}^{\underline{i}}-\bar{v}_j^i)d\theta|dx\\
\displaystyle & =C\|\varphi\|_{C^\alpha}l^\alpha\sum\limits_{i,j}\int_{x_{j-1}}^{x_{j+1}}|\int_0^1\partial_{\bar{\varrho}_\theta}\eta(\bar{v}_\theta)(\bar{\varrho}^{\underline{i}}-\bar{\varrho}_j^i)+\partial_{\bar{\omega}_\theta}\eta(\bar{v}_\theta)(\bar{\omega}^{\underline{i}}-\bar{\omega}_j^i)d\theta|dx,
\end{aligned}
\end{equation}
where $\bar{\varrho}_\theta=\bar{\varrho}^i_j+\theta(\bar{\varrho}^{\underline{i}}-\bar{\varrho}^i_j)=\bar{\varrho}^{\underline{i}}+(1-\theta)(\bar{\varrho}^i_j-\bar{\varrho}^{\underline{i}})$ and some constant $1>\alpha>\frac{8}{9}$. For simplicity, we choose $\eta_\theta=\eta(\bar{v}_\theta)$.
From the Lemma \ref{31hs} and \eqref{eq188}, using the H\"{o}lder inequality, we have
\begin{equation}\label{eq215}
\begin{array}{ll}
\displaystyle\sum\limits_{i,j}\int_{x_{j-1}}^{x_{j+1}}|\int_0^1\partial_{\bar{\varrho}_\theta}\eta_\theta(\bar{\varrho}^{\underline{i}}-\bar{\varrho}_j^i)d\theta|dx\\
\displaystyle\qquad \quad \leq C\sum\limits_{i,j}\int_{x_{j-1}}^{x_{j+1}}|\int_0^1(1+|\log\bar{\varrho}_\theta|)\bar{\varrho}_\theta^{\frac{\xi^2-|\xi|}{1-\xi^2}}(\bar{\varrho}^{\underline{i}}-\bar{\varrho}_j^i)d\theta|dx\\
\displaystyle\qquad \quad \leq C\sum\limits_{i,j}\int_{x_{j-1}}^{x_{j+1}}|\int_0^1(1+|\log\bar{\varrho}_\theta|)\bar{\varrho}_\theta^{\frac{\xi^2-|\xi|}{1-\xi^2}}(\bar{\varrho}^{\underline{i}}-\bar{\varrho}_j^i)d\theta|dx|_{\bar{\varrho}_\theta>1}\\
\displaystyle\qquad \quad\quad + C\sum\limits_{i,j}\int_{x_{j-1}}^{x_{j+1}}|\int_0^1(1+|\log\bar{\varrho}_\theta|)\bar{\varrho}_\theta^{\frac{\xi^2-|\xi|}{1-\xi^2}}(\bar{\varrho}^{\underline{i}}-\bar{\varrho}_j^i)d\theta|dx|_{\bar{\varrho}_\theta<1}\\
\displaystyle\qquad \quad \leq C(T,L)h^{-\frac{1}{2}}+ C\sum\limits_{i,j}\int_{x_{j-1}}^{x_{j+1}}|\bar{\varrho}^{\underline{i}}-\bar{\varrho}_j^i|^\frac{1}{4}dx\\
\displaystyle\qquad \quad \leq C(T,L)h^{-\frac{7}{8}},
\end{array}
\end{equation}
where $\frac{\xi^2-|\xi|}{1-\xi^2}\in(-\frac{1}{2},0)$ when $\xi\in(-\frac{1}{2},\frac{1}{2}).$
In addition, we give
\begin{equation}\label{eq217}
\begin{array}{ll}
\displaystyle\sum\limits_{i,j}\int_{x_{j-1}}^{x_{j+1}}|\int_0^1\partial_{\bar{\omega}_\theta}\eta(\bar{v}_\theta)(\bar{\omega}^{\underline{i}}-\bar{\omega}_j^i)d\theta|dx\\
\displaystyle\qquad \quad \leq C\sum\limits_{i,j}\int_{x_{j-1}}^{x_{j+1}}|\int_0^1{\bar{\varrho}_\theta}^{\frac{\xi^2-|\xi|}{1-\xi^2}}(\bar{\omega}^{\underline{i}}-\bar{\omega}_j^i)d\theta|dx\\
\displaystyle\qquad \quad \leq C\sum\limits_{i,j}\int_{x_{j-1}}^{x_{j+1}}|\int_0^1\bar{\varrho}_\theta^{\frac{\xi^2-|\xi|}{1-\xi^2}}(\bar{\omega}^{\underline{i}}-\bar{\omega}_j^i)d\theta|dx|_{\bar{\varrho}_\theta>1} \\
\displaystyle\qquad \quad\quad+ C\sum\limits_{i,j}\int_{x_{j-1}}^{x_{j+1}}|\int_0^1\bar{\varrho}_\theta^{\frac{\xi^2-|\xi|}{1-\xi^2}}(\bar{\omega}^{\underline{i}}-\bar{\omega}_j^i)d\theta|dx|_{\bar{\varrho}_\theta<1}\\
\displaystyle\qquad \quad \leq C\sum\limits_{i,j}\int_{x_{j-1}}^{x_{j+1}}|\bar{\omega}^{\underline{i}}-\bar{\omega}_j^i|dx+ C\sum\limits_{i,j}\int_{x_{j-1}}^{x_{j+1}}|\int_0^1\bar{\varrho}_\theta^{-\frac{1}{2}}|\bar{\omega}^{\underline{i}}-\bar{\omega}_j^i|d\theta dx\\
\displaystyle\qquad \quad \leq C(T,L)h^{-\frac{1}{2}}+C\sum\limits_{i,j}\int_{x_{j-1}}^{x_{j+1}}|\int_0^1\bar{\varrho}_\theta^{-\frac{1}{2}}|\bar{\omega}^{\underline{i}}-\bar{\omega}_j^i|d\theta dx.
\end{array}
\end{equation}
By Lemma \ref{31hs}, there yields
\begin{equation}\label{eq216}
\begin{aligned}
|\bar{\omega}^{\underline{i}}-\bar{\omega}_j^i|\leq C(|\bar{\varrho}^{\underline{i}}|^{\frac{3}{4}}+|\bar{\varrho}_j^i|^{\frac{3}{4}}).
\end{aligned}
\end{equation}
To get a more accurate estimate for the second terms in \eqref{eq217}, we divide following four cases.

{\bf Case \rmnum{1}. } If $0\leq\bar{\varrho}^{\underline{i}}-\bar{\varrho}_j^i\leq\bar{\varrho}_j^i$, then $|\bar{\omega}^{\underline{i}}-\bar{\omega}_j^i|\leq C|\bar{\varrho}_j^i|^{\frac{3}{4}}$, that is
\begin{equation}\label{eq218}
\begin{aligned}
|\bar{\varrho}_j^i|^{-\frac{1}{2}}\leq C|\bar{\omega}^{\underline{i}}-\bar{\omega}_j^i|^{-\frac{2}{3}}.
\end{aligned}
\end{equation}
We have
\begin{equation}\label{eq219}
\begin{aligned}
\sum\limits_{i,j}\int_{x_{j-1}}^{x_{j+1}}\int_0^1\bar{\varrho}_\theta^{-\frac{1}{2}}|\bar{\omega}^{\underline{i}}-\bar{\omega}_j^i|d\theta dx&\leq C\sum\limits_{i,j}\int_{x_{j-1}}^{x_{j+1}}|\bar{\omega}^{\underline{i}}-\bar{\omega}_j^i||\bar{\varrho}_j^i|^{-\frac{1}{2}}dx\\
&\leq C\sum\limits_{i,j}\int_{x_{j-1}}^{x_{j+1}}|\bar{\omega}^{\underline{i}}-\bar{\omega}_j^i|^{\frac{1}{3}}dx\leq Ch^{-\frac{5}{6}}.
\end{aligned}
\end{equation}

{\bf Case \rmnum{2}. } If $\bar{\varrho}_j^i\leq\bar{\varrho}^{\underline{i}}-\bar{\varrho}_j^i$, then $\bar{\varrho}^{\underline{i}}\leq2(\bar{\varrho}^{\underline{i}}-\bar{\varrho}_j^i)$, that is
\begin{equation}\label{eq220}
\begin{aligned}
|\bar{\varrho}^{\underline{i}}-\bar{\varrho}_j^i|^{-\frac{1}{2}}\leq C|\bar{\omega}^{\underline{i}}-\bar{\omega}_j^i|^{-\frac{2}{3}}.
\end{aligned}
\end{equation}
There is
\begin{equation}\label{eq221}
\begin{aligned}
\sum\limits_{i,j}\int_{x_{j-1}}^{x_{j+1}}\int_0^1\bar{\varrho}_\theta^{-\frac{1}{2}}|\bar{\omega}^{\underline{i}}-\bar{\omega}_j^i|d\theta dx&\leq C\sum\limits_{i,j}\int_{x_{j-1}}^{x_{j+1}}|\bar{\omega}^{\underline{i}}-\bar{\omega}_j^i||\bar{\varrho}^{\underline{i}}-\bar{\varrho}_j^i|^{-\frac{1}{2}}dx\\
&\leq
C\sum\limits_{i,j}\int_{x_{j-1}}^{x_{j+1}}|\bar{\omega}^{\underline{i}}-\bar{\omega}_j^i|^{\frac{1}{3}}dx\leq Ch^{-\frac{5}{6}}.
\end{aligned}
\end{equation}

{\bf Case \rmnum{3}. } If $0\leq\bar{\varrho}_j^i\leq\bar{\varrho}_j^i-\bar{\varrho}^{\underline{i}}\leq\bar{\varrho}^{\underline{i}}$, similarly to the Case \rmnum{1}, we have
\begin{equation}\label{eq222}
\begin{aligned}
\sum\limits_{i,j}\int_{x_{j-1}}^{x_{j+1}}\int_0^1\bar{\varrho}_\theta^{-\frac{1}{2}}|\bar{\omega}^{\underline{i}}-\bar{\omega}_j^i|d\theta dx&\leq C\sum\limits_{i,j}\int_{x_{j-1}}^{x_{j+1}}|\bar{\omega}^{\underline{i}}-\bar{\omega}_j^i||\bar{\varrho}^{\underline{i}}|^{-\frac{1}{2}}dx\\
&\leq C\sum\limits_{i,j}\int_{x_{j-1}}^{x_{j+1}}|\bar{\omega}^{\underline{i}}-\bar{\omega}_j^i|^{\frac{1}{3}}dx\leq Ch^{-\frac{5}{6}}.
\end{aligned}
\end{equation}

{\bf Case \rmnum{4}. } If $\bar{\varrho}^{\underline{i}}\leq\bar{\varrho}_j^i-\bar{\varrho}^{\underline{i}}$, similarly to the Case \rmnum{2}, we have
\begin{equation}\label{eq2221}
\begin{aligned}
\sum\limits_{i,j}\int_{x_{j-1}}^{x_{j+1}}\int_0^1\bar{\varrho}_\theta^{-\frac{1}{2}}|\bar{\omega}^{\underline{i}}-\bar{\omega}_j^i|d\theta dx&\leq C\sum\limits_{i,j}\int_{x_{j-1}}^{x_{j+1}}|\bar{\omega}^{\underline{i}}-\bar{\omega}_j^i||\bar{\varrho}_j^i-\bar{\varrho}^{\underline{i}}|^{-\frac{1}{2}}dx\\
&\leq
C\sum\limits_{i,j}\int_{x_{j-1}}^{x_{j+1}}|\bar{\omega}^{\underline{i}}-\bar{\omega}_j^i|^{\frac{1}{3}}dx\leq Ch^{-\frac{5}{6}}.
\end{aligned}
\end{equation}
From the above estimates \eqref{eq216}-\eqref{eq2221}, we obtain that
\begin{equation}\label{eq223}
\begin{array}{ll}
\displaystyle\sum\limits_{i,j}\int_{x_{j-1}}^{x_{j+1}}|\int_0^1\partial_{\bar{\omega}_\theta}\eta(\bar{v}_\theta)(\bar{\omega}^{\underline{i}}-\bar{\omega}_j^i)d\theta|dx\\
\displaystyle\qquad \quad \leq C(T,L)h^{-\frac{1}{2}}+C\sum\limits_{i,j}\int_{x_{j-1}}^{x_{j+1}}|\int_0^1\bar{\varrho}_\theta^{-\frac{1}{2}}|\bar{\omega}^{\underline{i}}-\bar{\omega}_j^i|d\theta dx\\
\displaystyle\qquad \quad \leq Ch^{-\frac{1}{2}}+Ch^{-\frac{5}{6}}\leq C(T,L)h^{-\frac{5}{6}}.
\end{array}
\end{equation}
Hence, we get
\begin{equation}\label{eq1214}
\begin{array}{ll}
|D_2(\varphi)| \leq C\|\varphi\|_{C^\alpha}l^\alpha (h^{-\frac{7}{8}}+h^{-\frac{5}{6}})\leq C\|\varphi\|_{C^\alpha}l^{\alpha-\frac{8}{9}}.
\end{array}
\end{equation}

{\bf step 3. } From \eqref{eq206}, \eqref{eq207} and \eqref{eq213}, we have
\begin{equation}\label{eq224}
\begin{aligned}
\|A+D_1+S\|_{(C_0(\Omega))^*}\leq C,
\end{aligned}
\end{equation}
where $C$ depends only on $\Omega$ and $\eta$. Owing to the above estimates, we can apply Lemma \ref{41hs} to get the $H^{-1}$ compactness.
For $1<p_1<2$, $(C_0(\Omega))^*\hookrightarrow W^{-1,p_1}(\Omega)$ is compact by the embedding theorem. That means $A+D_1+S$ is compact in $W^{-1,p_1}(\Omega)$.
For $0<\alpha<1-\frac{2}{p_2}$, $W_0^{1,p_2}(\Omega)\subset C^\alpha_0(\Omega)$ by the Sobolev theorem, that is
\begin{equation}\label{eq225}
\begin{aligned}
|D_2(\varphi)|\leq C\|\varphi\|_{C_0^\alpha(\Omega)}l^{\alpha-\frac{8}{9}}\leq C\|\varphi\|_{W_0^{1,p_2}(\Omega)}l^{\alpha-\frac{8}{9}},
\end{aligned}
\end{equation}
where $p_2>\frac{2}{1-\alpha}$.
We consider the duality that
\begin{equation}\label{eq226}
\begin{aligned}
\|D_2\|_{W^{-1,p_3}(\Omega)}\leq Cl^{\alpha-\frac{8}{9}}\rightarrow 0 \ \ \mbox{as}\ \ l\rightarrow 0, \  \mbox{for} \ \frac{8}{9}<\alpha<1\  \mbox{and}\ 1<p_3<\frac{2}{1+\alpha},
\end{aligned}
\end{equation}
then $D_2$ is compact in $W^{-1,p_3}(\Omega)$.
Therefore, $A+D+S=A+D_1+D_2+S$ is compact in $W^{-1,p}(\Omega)$, where $1\leq p\leq \min(p_1,p_3)$.

Furthermore, we have $\partial_t\eta(\bar{v})+\partial_xq(\bar{v})-B$ is bounded in $W^{-1,\infty}(\Omega)$, because of the uniform bound of $\bar{v}$ and the continuity of $\eta$ and $q$. And $\partial_t\eta(\bar{v})+\partial_xq(\bar{v})-B$ is bounded in $W^{-1,p_4}(\Omega)$ for $p_4>1$ by the boundedness of $\Omega$. Thus, $A+D+S$ is bounded in $W^{-1,p_4}$. From Lemma \ref{41hs}, we have $$A+D+S \ \mbox{is compact in}\  H^{-1}_{loc}(\Pi),$$ that is,
\begin{equation}\label{eq227}
\begin{aligned}
\partial_t\eta(\bar{v})+\partial_xq(\bar{v})-B  \ \mbox{is compact in }\ H^{-1}_{loc}(\Pi).
\end{aligned}
\end{equation}
Then $\partial_t\eta(\bar{v})+\partial_xq(\bar{v})$ is compact in $H^{-1}_{loc}(\Pi)$ for $\xi\in (-\frac{1}{2},\frac{1}{2})$.
\end{proof}

So far, combining Lemma \ref{31hs} with Theorem \ref{theorem41}, we have the compactness framework for the approximate solutions $\bar{v}^l(x,t)$ defined in section 3.

\section{Convergence and existence}\label{section5}

\setcounter{section}{5}

\setcounter{equation}{0}

\noindent  In this section, we use the compensated compactness framework to get a convergent subsequence. Then we show that the limit of this sequence is a mechanic entropy solution.

\vskip 0.1in

\begin{theorem}\label{theorem51} Suppose that $(\bar{\varrho}^l, \bar{\omega}^l)$ satisfy the conditions in Lemma \ref{31hs} and Theorem \ref{theorem41}. Then

$(i)$ there exists a convergent subsequence $\bar{v}^l(x,t)=(\bar{\varrho}^l(x,t),\bar{\omega}^l(x,t))$ such that
\begin{equation}\label{eq228}
\begin{aligned}
(\bar{\varrho}^l(x,t),\bar{\omega}^l(x,t))\rightarrow(\varrho(x,t),\omega(x,t)), \qquad \ \ \mbox{a.e. }
\end{aligned}
\end{equation}
for some positive constant $C=C(T,M)$, where the vector function $(\varrho(x,t),\omega(x,t))$ is a bounded measurable and 
satisfies
\begin{equation}\label{eq229}
\begin{aligned}
0\leq\varrho(x,t)\leq C,\quad |\omega(x,t)|\leq\varrho(x,t)(C+|\log\varrho(x,t)|).
\end{aligned}
\end{equation}

$(ii)$ $(\varrho(x,t),\omega(x,t))$ is an entropy solution of \eqref{eq120}, i.e. $$(\rho(x,t),m(x,t))=(\frac{\varrho(x,t)}{x^{N-1}},\frac{\omega(x,t)}{x^{N-1}})$$ is an entropy solution of \eqref{eq113} with \eqref{eq115}, \eqref{eq114}. Namaly, $$(\rho(\vec{x},t),\vec{m}(\vec{x},t))=(\rho(|\vec{x}|,t),m(|\vec{x}|,t)\frac{\vec{x}}{|\vec{x}|})$$ is a spherically symmetric entropy solution to the multi-dimensional Euler-Poisson equation \eqref{eq111} with  initial-boundary data \eqref{eq112}, \eqref{eq115} and \eqref{eq114}.
\end{theorem}

\vskip 0.1in

\begin{proof}
By the Theorem \ref{theorem0}, Lemma \ref{31hs} and Theorem \ref{theorem41}, we obtain a convergent subsequence $\bar{v}^l(x,t)$ such that
\begin{equation}\label{eq230}
\begin{aligned}
(\bar{\varrho}^l(x,t),\bar{\omega}^l(x,t))\rightarrow(\varrho(x,t),\omega(x,t)) \qquad \ \ \mbox{a.e. }
\end{aligned}
\end{equation}
and $(\varrho(x,t),\omega(x,t))$ satisfies
\begin{equation}\label{eq231}
\begin{aligned}
0\leq\varrho(x,t)\leq C,\quad |\omega(x,t)|\leq\varrho(x,t)(C+|\log\varrho(x,t)|)
\end{aligned}
\end{equation}
for some constant $C=C(T,M)>0$ and $0\leq t \leq T$.

For any function $\varphi, \psi\ge0\in C_0^\infty({\Pi}_T)$, 
 there exist $L>0$ and $T>0$ such that $$\mbox{supp}\varphi, \mbox{supp}\psi\subset[1,L)\times[0,T).$$ Then we need to prove that the function $(\varrho(x,t),\omega(x,t))$ satisfies the standard notion of the weak entropy solution in Definition \ref{def1} for any test function $\varphi(x,t)$ and $\psi(x,t)\ge0$. And $(\eta_e, q_e)$ is the mechanical entropy and the corresponding flux as defined in \eqref{eq13391}.

Firstly, we claim that
\begin{equation}\label{eq232}
\begin{aligned}
\iint\limits_{\Pi_T}(\varrho\varphi_t+\omega\varphi_x)dxdt+\int_1^{+\infty}\varrho_0(x)\varphi(x,0)dx=0.
\end{aligned}
\end{equation}
Indeed, for the convergent subsequence, we denote that
\begin{equation}\label{eq233}
\begin{array}{ll}
\displaystyle\iint\limits_{\Pi_T}(\bar{\varrho}^l\varphi_t+\bar{\omega}^l\varphi_x)dxdt+\int_1^{+\infty}\bar{\varrho}^l_0(x)\varphi(x,0)dx\\
=\displaystyle\sum\limits_{i,j}\int_{x_{j-1}}^{x_{j+1}}(\underline{\varrho}^{l\underline{i}}-\bar{\varrho}_j^i)\varphi^idx+\iint\limits_{\Pi_T}(\bar{\varrho}^l-\underline{\varrho}^l)\varphi_tdxdt+\iint\limits_{\Pi_T}(\bar{\omega}^l-\underline{\omega}^l)\varphi_xdxdt\\
\displaystyle\qquad\qquad +\int_1^{+\infty}(\bar{\varrho}^l(x,0)-\underline{\varrho}^l(x,0))\varphi(x,0)dx\\
=\Rmnum{1}+\Rmnum{2}+\Rmnum{3}+\Rmnum{4},
\end{array}
\end{equation}
where $\underline{\varrho}^{l\underline{i}}=\underline{\varrho}^{h}(x,t_i-0)$ and $\varphi^i=\varphi(x,t_i)$.
There exists
\begin{equation}\label{eq234}
\begin{array}{ll}
|\Rmnum{1}|=\displaystyle|\sum\limits_{i,j}\int_{x_{j-1}}^{x_{j+1}}(\underline{\varrho}^{l\underline{i}}-\bar{\varrho}_j^i)\varphi^idx|\\
\displaystyle\qquad \leq|\sum\limits_{i,j}\int_{x_{j-1}}^{x_{j+1}}(\varrho^{l\underline{i}}-\bar{\varrho}^{l\underline{i}})\varphi^idx|+|\sum\limits_{i,j}\int_{x_{j-1}}^{x_{j+1}}(\bar{\varrho}^{l\underline{i}}-\bar{\varrho}_j^i)\varphi^idx|\\
\displaystyle\qquad \leq C\sum\limits_{i,j}\int_{x_{j-1}}^{x_{j+1}}l^\beta\|\varphi\|_\infty dx+|\sum\limits_{i,j}\int_{x_{j-1}}^{x_{j+1}}(\bar{\varrho}^{l\underline{i}}-\bar{\varrho}_j^i)\varphi_j^idx|\\
\displaystyle\qquad\quad+|\sum\limits_{i,j}\int_{x_{j-1}}^{x_{j+1}}(\bar{\varrho}^{l\underline{i}}-\bar{\varrho}_j^i)(\varphi^i-\varphi_j^i)dx|\\
\displaystyle\qquad \leq C(L)\|\varphi\|_\infty l^{\beta}\frac{T}{h}+0+C(T,L)\|\varphi\|_{C_0^\alpha}l^\alpha h^{-\frac{1}{2}}\\
\qquad \leq C(T,L)[l^{\beta-1}(1+|\log l|)+l^{\alpha-\frac{1}{2}}(1+|\log l|)^{\frac{1}{2}}]\rightarrow 0 \qquad \ \ \mbox{as} \quad l\rightarrow0,
\end{array}
\end{equation}
for $\frac{8}{9}<\alpha<1$ and $3\leq\beta\leq4$, where $\varphi_j^i=\varphi(x_j,t_i)$. Also,
\begin{equation}\label{eq235}
\begin{array}{ll}
\displaystyle|\Rmnum{2}|=|\iint\limits_{\Pi_T}(\bar{\varrho}^l-\underline{\varrho}^l)\varphi_tdxdt|\leq C(L,T)l^\beta\rightarrow 0,\\
\end{array}
\end{equation}
and
\begin{equation}\label{eq236}
\begin{array}{ll}
\displaystyle|\Rmnum{3}|=|\iint\limits_{\Pi_T}(\bar{\omega}^l-\underline{\omega}^l)\varphi_xdxdt|=|\sum\limits_{i,j}\int_{t_i}^{t_{i+1}}\int_{x_{j-1}}^{x_{j+1}}g_2(\underline{v}^l)(t-t_i)\varphi_xdxdt|\\
\displaystyle\leq h|\sum\limits_{i,j}\int_{t_i}^{t_{i+1}}\int_{x_{j-1}}^{x_{j+1}}(\frac{N-1}{x}\underline{\varrho}^l-\frac{\underline{\varrho}^l}{x^{N-1}}\int_0^x\underline{\varrho}^l(s,t)ds)\varphi_xdxdt|\\
 \displaystyle\leq Ch\iint\limits_{\Pi_T}|\varphi_x|dxdt
 \leq Cl^\frac{1}{2}\rightarrow 0\qquad \ \ \mbox{as} \quad l\rightarrow0.
\end{array}
\end{equation}
From \eqref{eq154} and \eqref{eq188}, we have
\begin{equation}\label{eq237}
\begin{array}{ll}
\displaystyle|\Rmnum{4}|=|\int_1^{+\infty}(\bar{\varrho}^l(x,0)-\underline{\varrho}^l(x,0))\varphi(x,0)dx|\\
\displaystyle\qquad \leq|\sum\limits_{j}\int_{x_{j-1}}^{x_{j+1}}(\bar{\varrho}^l(x,0)-\underline{\varrho}^l(x,0))\varphi(x_j,0)dx|\\
\displaystyle\qquad \qquad+|\sum\limits_{j}\int_{x_{j-1}}^{x_{j+1}}(\bar{\varrho}^l(x,0)-\underline{\varrho}^l(x,0))(\varphi(x,0)-\varphi(x_j,0))dx|\\
\displaystyle\qquad\leq 0+Cl\|\varphi\|_{C^1}\sum\limits_{j}\int_{x_{j-1}}^{x_{j+1}}|\bar{\varrho}^l(x,0)-\underline{\varrho}^l(x,0)|dx\\
\qquad \leq Clh^{-\frac{1}{2}}\leq Cl^{\frac{1}{4}}\rightarrow 0\ \ \mbox{as} \quad l\rightarrow0.
\end{array}
\end{equation}
From the above inequalities, we have 
\begin{equation}\label{eq238}
\begin{aligned}
\iint\limits_{\Pi_T}(\bar{\varrho}^l\varphi_t+\bar{\omega}^l\varphi_x)dxdt+\int_1^{+\infty}\bar{\varrho}^l(x,0)\varphi(x,0)dx\rightarrow 0\ \ \mbox{as} \quad l\rightarrow0.
\end{aligned}
\end{equation}
In addition, notice that
\begin{equation}\label{eq239}
\begin{aligned}
|\int_1^{+\infty}&\bar{\varrho}^l(x,0)\varphi(x,0)dx-\int_1^{+\infty}\varrho_0(x)\varphi(x,0)dx|\\
&\leq l^\beta\int_1^{+\infty}|\varphi(x,0)|dx\rightarrow 0\ \ \mbox{as} \quad l\rightarrow0.
\end{aligned}
\end{equation}
Thus, combining \eqref{eq230}, \eqref{eq238} - \eqref{eq239} with the dominated convergence theorem, we get \eqref{eq232}.

From \eqref{eq232}, we know
\begin{equation}\label{eq241}
\begin{aligned}
\omega(x,\cdot) \xrightarrow{*} 0\quad \ \mbox{weakly },\qquad \ \ \mbox{as } \ \ x\rightarrow1.
\end{aligned}
\end{equation}
Combining \eqref{eq241} with the function $\omega$ which has a well-defined trace in $L^\infty$, we use the trace theorem in the theory of divergence-measure fields to obtain
\begin{equation}\label{eq242}
\begin{aligned}
\omega|_{x=1}= 0
\end{aligned}
\end{equation}
in the sense of traces introduced by Chen-Frid in \cite{CF}. It means \eqref{eq08} holds.

Now, we show that
\begin{equation}\label{eq240}
\begin{array}{ll}
\displaystyle\iint\limits_{\Pi_T}\big[\omega\varphi_t+(\frac{\omega^2}{\varrho}+\varrho)\varphi_x+(\frac{N-1}{x}\varrho-\frac{\varrho}{x^{N-1}}\int_1^x\varrho(s, t)ds)\varphi\big]dxdt\\
\displaystyle\qquad \quad+\int_1^{+\infty}\omega_0(x)\varphi(x,0)dx=0.
\end{array}
\end{equation}
By the Green formula, there exists
\begin{equation}\label{eq243}
\begin{aligned}
\displaystyle\iint\limits_{\Pi_T}\big(\underline{\omega}^l\varphi_t+f_2(\underline{v}^l)\varphi_x\big)dxdt+\int_1^{+\infty}\underline{\omega}^l(x,0)\varphi(x,0)dx\\
=\sum\limits_{i,j}\int_{x_{j-1}}^{x_{j+1}}(\underline{\omega}^{l\underline{i}}-\bar{\omega}^i_j)\varphi^idx,
\end{aligned}
\end{equation}
where $f_2$ defined in \eqref{eq122}. Since $\bar{\omega}^l=\omega^l$ and $g_2$ defined in \eqref{eq122}, we have
\begin{equation}\label{eq244}
\begin{aligned}
\displaystyle\iint\limits_{\Pi_T}\big(\bar{\omega}^l\varphi_t+f_2(\bar{v}^l)\varphi_x+g_2(\bar{v}^l)\varphi\big) dxdt+\int_1^{+\infty}\bar{\omega}^l(x,0)\varphi(x,0)dx\\
\qquad=J_1(\varphi)+J_2(\varphi)+J_3(\varphi)+J_4(\varphi),
\end{aligned}
\end{equation}
where
\begin{equation}\label{eq245}
\begin{aligned}
\displaystyle &J_1(\varphi)=\iint\limits_{\Pi_T}\big(f_2(\bar{v}^l)-f_2(v^l))\varphi_x+(g_2(\bar{v}^l)-g_2(v^l))\varphi\big)dxdt,\\
\displaystyle &J_2(\varphi)=\iint\limits_{\Pi_T}\big(f_2(v^l)-f_2(\underline{v}^l))\varphi_x+(g_2({v}^l)-g_2(\underline{v}^l))\varphi+(\omega^l-\underline{\omega}^l)\varphi_t\big)dxdt,\\ \displaystyle &J_3(\varphi)=\int_1^{+\infty}\big(\bar{\omega}^l(x,0)-\underline{\omega}^l(x,0)\big)\varphi(x,0)dx,\\
\displaystyle &J_4(\varphi)=\sum\limits_{i,j}\int_{x_{j-1}}^{x_{j+1}}(\underline{\omega}^{l\underline{i}}-\bar{\omega}^i_j)\varphi^idx+\iint\limits_{\Pi_T}g_2(\underline{v}^l)\varphi dxdt.
\end{aligned}
\end{equation}
Since $3\leq\beta\leq4$, we have
\begin{equation}\label{eq246}
\begin{aligned}
 |J_1(\varphi)|&\displaystyle\leq\iint\limits_{\Pi_T}\big((|\frac{(\bar{\omega}^l)^2}{\bar{\varrho}^l}-\frac{(\omega^l)^2}{\varrho^l}|+|\bar{\varrho}^l-\varrho^l|)\varphi_x+|\frac{N-1}{x}(\bar{\varrho}^l-\varrho^l)\\
&\displaystyle \qquad -\frac{1}{x^{N-1}}(\bar{\varrho}^l\int_1^x\bar{\varrho}^l(s,t)ds-\varrho^l\int_1^x\varrho^l(s,t)ds)|\varphi \big)dxdt|_{\varrho^l\leq l^\beta}\\
&\displaystyle \leq C\iint\limits_{\Pi_T}\big(\bar{\varrho}^l(1+|\log\bar{\varrho}^l|)^2+l^\beta+Ml^\beta\big)dxdt|_{\varrho^l\leq l^\beta}\\
&\displaystyle\leq C(L,T)l^\beta(1+|\log l|)^2\frac{1}{h}\leq Cl^{\beta-1}(1+|\log l|)^3  \ \ \rightarrow0 \quad  \ \ \mbox{as } \ \ l\rightarrow0,
\end{aligned}
\end{equation}
\begin{equation}\label{eq247}
\begin{aligned}
\displaystyle|J_2(\varphi)|&\leq\sum\limits_{i,j}\int_{t_i}^{t_{i+1}}\int_{x_{j-1}}^{x_{j+1}}\big((|\frac{({\omega}^l)^2}{{\varrho}^l}-\frac{(\underline{\omega}^l)^2}{\underline{\varrho}^l}|+|{\varrho}^l-\underline{\varrho}^l|)\varphi_x+|\frac{N-1}{x}({\varrho}^l-\underline{\varrho}^l)\\
&\displaystyle\qquad -\frac{1}{x^{N-1}}({\varrho}^l\int_1^x{\varrho}^l(s,t)ds-\underline{\varrho}^l\int_1^x\underline{\varrho}^l(s,t)ds)|\varphi+(\omega^l-\underline{\omega}^l)\varphi_t \big)dxdt\\
&\displaystyle \leq C\iint\limits_{\Pi_T}(\omega^l-\underline{\omega}^l)\big[(1+|\log\underline{\varrho}^l|)|\varphi_x|+|\varphi_t|\big]dxdt\\
&\displaystyle= C\iint\limits_{\Pi_T}h\frac{|g_2(\underline{v}^l)|}{\underline{\varrho}^l}\underline{\varrho}^l\big[(1+|\log\underline{\varrho}^l|)|\varphi_x|+|\varphi_t|\big]dxdt\\
& \leq Ch(1+|\log l^\beta|)+Ch\leq Cl(1+|\log l|)  \ \ \rightarrow0 \qquad  \ \ \mbox{as } \ \ l\rightarrow0,
\end{aligned}
\end{equation}
\begin{equation}\label{eq248}
\begin{aligned}
\displaystyle|J_3(\varphi)|&\leq\int_1^{+\infty}\big(\bar{\omega}^l(x,0)-\underline{\omega}^l(x,0)\big)\varphi(x_j,0)+\big(\bar{\omega}^l(x,0)-\underline{\omega}^l(x,0)\big)\big(\varphi(x,0)-\varphi(x_j,0)\big)dx,\\
&\displaystyle\leq 0+Cl\sum\limits_{j}\int_{x_{j-1}}^{x_{j+1}}{\varrho}^l(1+|\log{\varrho}^l|)(x,0)|\frac{\varphi(x,0)-\varphi(x_j,0)}{l}|dx \\
& \leq Cl^\frac{1}{2}\|\varphi\|_{C^1} \ \ \rightarrow0 \qquad\qquad  \ \ \mbox{as } \ \ l\rightarrow0,
\end{aligned}
\end{equation}
\begin{equation}\label{eq249}
\begin{aligned}
\displaystyle J_4&(\varphi)=\sum\limits_{i,j}\int_{x_{j-1}}^{x_{j+1}}(\underline{\omega}^{l\underline{i}}-\bar{\omega}^i_j)\varphi^i+({\omega}^i_j-{\omega}^{l\underline{i}})\varphi ^i_jdx+\iint\limits_{\Pi_T}g_2(\underline{v}^l)\varphi dxdt\\
\displaystyle &=\sum\limits_{i,j}\int_{x_{j-1}}^{x_{j+1}}({\omega}^{l\underline{i}}-\bar{\omega}^i_j)(\varphi^i-\varphi^i_j)dx+\sum\limits_{i,j}\int_{t_i}^{t_{i+1}}\int_{x_{j-1}}^{x_{j+1}}g_2(\underline{v}^l)(\varphi-\varphi^i)dxdt\\
\displaystyle&\qquad \qquad +\sum\limits_{i,j}\int_{t_i}^{t_{i+1}}\int_{x_{j-1}}^{x_{j+1}}\big(g_2(\underline{v}^l)-g_2(\underline{v}^{l\underline{i}})\big)\varphi^i dxdt\\
& =J^1_4(\varphi)+J^2_4(\varphi)+J^3_4(\varphi).
\end{aligned}
\end{equation}
From \eqref{eq188}, we have
\begin{equation}\label{eq250}
\begin{aligned}
\displaystyle|J^1_4(\varphi)|&=|\sum\limits_{i,j}\int_{x_{j-1}}^{x_{j+1}}l({\omega}^{l\underline{i}}-\bar{\omega}^i_j)\frac{(\varphi^i-\varphi^i_j)}{l}dx|\\
&\leq Cl\|\varphi\|_{C^1}\sum\limits_{i,j}\int_{x_{j-1}}^{x_{j+1}}|{\omega}^{l\underline{i}}-{\omega}^i_j|dx\\
 &\leq Clh^{-\frac{1}{2}}\leq Cl^\frac{1}{2}(1+|\log l|)^\frac{1}{2}   \ \ \rightarrow0 \qquad\qquad  \ \ \mbox{as } \ \ l\rightarrow0.
\end{aligned}
\end{equation}
By Lemma \ref{31hs}, we get
\begin{equation}\label{eq251}
\begin{array}{ll}
\displaystyle|J^2_4(\varphi)|\leq\sum\limits_{i,j}\int_{t_i}^{t_{i+1}}\int_{x_{j-1}}^{x_{j+1}}h|g_2(\underline{v}^l)\frac{\varphi-\varphi^i}{t-t_i}|dxdt\\
\displaystyle\qquad \quad \leq Ch\sum\limits_{i,j}\int_{t_i}^{t_{i+1}}\int_{x_{j-1}}^{x_{j+1}}|\frac{N-1}{x}\underline{\varrho}^l-\frac{\underline{\varrho}^l}{x^{N-1}}\int_1^x\underline{\varrho}^l(s,t)ds||\frac{\varphi-\varphi^i}{t-t_i}|dxdt\\
\displaystyle\qquad \quad \leq Ch\sum\limits_{i,j}\int_{t_i}^{t_{i+1}}\int_{x_{j-1}}^{x_{j+1}}|C+C(M)\underline{\varrho}^l||\frac{\varphi-\varphi^i}{t-t_i}|dxdt\\
\qquad \quad \leq Cl\|\varphi\|_{C^1}\ \ \rightarrow0 \qquad\qquad  \ \ \mbox{as } \ \ l\rightarrow0,
\end{array}
\end{equation}
and
\begin{equation}\label{eq252}
\begin{aligned}
\displaystyle|J^3_4(\varphi)|&\leq\sum\limits_{i,j}\int_{t_i}^{t_{i+1}}\int_{x_{j-1}}^{x_{j+1}}\big|\big(g_2(\underline{v}^l)-g_2(\underline{v}^{l\underline{i}})\big)\varphi^i \big| dxdt\\
\displaystyle&\leq \sum\limits_{i,j}\int_{t_i}^{t_{i+1}}\int_{x_{j-1}}^{x_{j+1}}|\frac{N-1}{x}(\underline{\varrho}^l-\underline{\varrho}^{l\underline{i}})\\
&\qquad\quad -\frac{1}{x^{N-1}}(\underline{\varrho}^l\int_1^x\underline{\varrho}^l(s,t)ds-\underline{\varrho}^{l\underline{i}}\int_1^x\underline{\varrho}^{l\underline{i}}(s,t)ds)||\varphi^i|dxdt\\
\displaystyle&\leq C\|\varphi\|_\infty\sum\limits_{i,j}\int_{t_i}^{t_{i+1}}\int_{x_{j-1}}^{x_{j+1}}(|{\varrho}^l-{\varrho}^{l\underline{i}}|+|\int_1^x{\varrho}^l(s,t)-{\varrho}^{l\underline{i}}(s,t)ds|)dxdt\\
\displaystyle&\leq C\big(\sum\limits_{i,j}\int_{t_i}^{t_{i+1}}\int_{x_{j-1}}^{x_{j+1}}|{\varrho}^l-{\varrho}^{l\underline{i}}|^2dxdt\big)^{\frac{1}{2}}\\
&\qquad\quad +C\sum\limits_{i,j}\int_{t_i}^{t_{i+1}}\int_{x_{j-1}}^{x_{j+1}}|\int_1^x{\varrho}^l(s,t)-{\varrho}^{l\underline{i}}(s,t)ds|dxdt.
\end{aligned}
\end{equation}

In the rectangle $\Pi_i=[x_{j-1},x_{j+1}]\times[t_{i-1},t_i)$, $i+j$ is even, $(\varrho(x,t),\omega(x,t))$ is the Riemann solution of the Riemann problem as described in \eqref{eq165} $-$ \eqref{eq167}, 
with $(\varrho^l, \omega^l)$ and $(\varrho^r, \omega^r)$ are the left state and right state, respectively. Since the monotonicity of rarefaction waves, we regard its intermediate state as some appropriate constants. We call both of these appropriate constants (if they exist) and the intermediate state between two waves as the intermediate constant state $(\varrho^m,\omega^m)$. 
We choose $\lambda_i=\frac{\omega}{\varrho}\pm1\leq C(T,M)+\beta|log l|+1$  and $h$ defined as \eqref{eq148}. When $0<l\leq l_0(T)$, the ratio of length of the interval of left state or right state and $2l$ are both bigger than $\frac{1}{3}$. For any given $\delta$, there are two cases needed to be considered.

{\bf Case 1. } The ratio of length of the interval of intermediate constant state and $2l$ is smaller than $\delta$. 
By Lemma \ref{3hs}, we have
\begin{equation}\label{eq254}
\begin{array}{ll}
&\displaystyle\sum\limits_{i,j}\int_{t_i}^{t_{i+1}}\int_{x_{j-1}}^{x_{j+1}}|{\varrho}^l-{\varrho}^{l\underline{i}}|^2dxdt\\
&\displaystyle\leq C\sum\limits_{i,j}\int_{t_i}^{t_{i+1}}l\big[\delta(\varrho^l-\varrho^m)^2+\delta(\varrho^r-\varrho^m)^2+(\varrho^l-\varrho^r)^2\big]dt\\
&\displaystyle\leq Cl\delta+C\sum\limits_{i,j}\int_{t_i}^{t_{i+1}}l\int_{x_{j-1}}^{x_{j+1}}|{\varrho}^{l\underline{i}}-{\varrho}^i_j|^2dxdt\\
&\leq C \delta+Ch\leq C\delta+Cl.
\end{array}
\end{equation}

{\bf Case 2. } The ratio of the interval of intermediate constant state and $2l$ is bigger than $\delta$, by Lemma \ref{3hs}, Lemma \ref{6hs} and Lemma \ref{31hs}, there exists
\begin{equation}\label{eq255}
\begin{array}{ll}
\displaystyle\sum\limits_{i,j}\int_{t_i}^{t_{i+1}}\int_{x_{j-1}}^{x_{j+1}}|{\varrho}^l-{\varrho}^{l\underline{i}}|^2dxdt\leq Ch\sum\limits_{i,j}\int_{x_{j-1}}^{x_{j+1}}\sum|\epsilon({\varrho}^{l\underline{i}})|^2dx\\
\displaystyle\quad\quad \leq Cl\delta^{-1}\sum\limits_{i,j}\int_{x_{j-1}}^{x_{j+1}}({\varrho}^{l\underline{i}}-\varrho_j^i)^2dx\leq Cl\delta^{-1},
\end{array}
\end{equation}
where $\sum|\epsilon({\varrho}^{l\underline{i}})|$ denotes the jump strengths of $\varrho^l(x,t)$ across shock waves in $\Pi_i$.

Thus, there is
\begin{equation}\label{eq256}
\begin{array}{ll}
\displaystyle\big(\sum\limits_{i,j}\int_{t_i}^{t_{i+1}}\int_{x_{j-1}}^{x_{j+1}}|{\varrho}^l-{\varrho}^{l\underline{i}}|^2dxdt\big)^\frac{1}{2}\leq C (\delta+l+l\delta^{-1})^\frac{1}{2}\leq C\delta^\frac{1}{2}\ \ \ \ \mbox{as }  \ l\rightarrow0,\\
\end{array}
\end{equation}
for any small $\delta>0.$

Because of the nonlocal influence of the gravitational field, from \eqref{eq120}$_1$, \eqref{eq126} and \eqref{eq174}, the remaining term of \eqref{eq252} can be estimated by 
\begin{equation}\label{eq258}
\begin{array}{ll}
\displaystyle\sum\limits_{i,j}\int_{t_i}^{t_{i+1}}\int_{x_{j-1}}^{x_{j+1}}|\int_1^x{\varrho}^l(s,t)-{\varrho}^{l\underline{i}}(s,t)ds|dxdt\\
\displaystyle\qquad \quad = \sum\limits_{i,j}\int_{t_i}^{t_{i+1}}\int_{x_{j-1}}^{x_{j+1}}|\int_1^x\int_t^{ih-0}({\varrho}^l)_\tau(s,\tau)d\tau ds|dxdt\\
\displaystyle\qquad \quad =\sum\limits_{i,j}\int_{t_i}^{t_{i+1}}\int_{x_{j-1}}^{x_{j+1}}|\int_t^{ih-0}\int_1^x({\omega}^l)_s(s,\tau)dsd\tau |dxdt\\
\displaystyle\qquad \quad =\sum\limits_{i,j}\int_{t_i}^{t_{i+1}}\int_{x_{j-1}}^{x_{j+1}}|\int_t^{ih-0}{\omega}^l(x,\tau)d\tau |dxdt\\
\qquad \quad \leq Ch\leq C\frac{l}{1+|\log l|}.
\end{array}
\end{equation}
Therefore, we obtain
\begin{equation}\label{eq259}
\begin{array}{ll}
|J^3_4(\varphi)|\leq (C \delta+Cl+Cl\delta^{-1}+C\frac{l}{1+|\log l|})^\frac{1}{2}\leq C\delta^\frac{1}{2} \ \ \ \ \mbox{as } \ l\rightarrow0,\\
\end{array}
\end{equation}
for any small $\delta$.

Using the dominated convergence theorem and
\begin{equation}\label{eq260}
\begin{array}{ll}
\displaystyle\int_1^{+\infty}\big({\bar{\omega}}^l(x,0)-{\omega}_0(x)\big)\varphi(x,0)dx=0,
\end{array}
\end{equation}
we conclude \eqref{eq240}.

Finally, we show the entropy inequality:
\begin{equation}\label{eq261}
\begin{array}{ll}
\displaystyle\iint\limits_{\Pi_T}&\big[\eta_e(v)\psi_t+q_e(v)\psi_x+\big(\frac{N-1}{x}\varrho-\frac{\varrho}{x^{N-1}}\int_1^x\varrho(s,t)ds\big)\eta_{e\omega}(v)\psi\big]dxdt\\
&+\int_1^{+\infty}\eta_e(v_0(x))\psi(x,0)dx\geq0.
\end{array}
\end{equation}

 Similar to $A(\psi)$, $B(\psi)$, $D(\psi)$ and $S(\psi)$ in step 2 of the section 4, we define the identity as:
\begin{equation}\label{eq262}
\begin{aligned}
\iint\limits_{\Pi_T}(\eta_e(\bar{v}^l)\psi_t+q_e(\bar{v}^l)\psi_x)dxdt=A(\psi)+B(\psi)+D(\psi)+S(\psi).
\end{aligned}
\end{equation}
From \eqref{eq199} and \eqref{eq200}, we know
\begin{equation}\label{eq264}
\begin{array}{ll}
\displaystyle A(\psi)=\sum\limits_{i,j}\int_{x_{j-1}}^{x_{j+1}}(\eta_e(\underline{v}^{l\underline{i}})-\eta_e(v^{l\underline{i}}))\psi^i(x)dx+\sum\limits_{i,j}\int_{x_{j-1}}^{x_{j+1}}(\eta_e(v^{l\underline{i}})-\eta_e(\bar{v}^{l\underline{i}}))\psi^i(x)dx,\\
\displaystyle\quad \geq- Cl^\frac{\beta}{2}\frac{1}{h}-h\sum\limits_{i,j}\int_{x_{j-1}}^{x_{j+1}}\int_0^1\eta_{e\omega}(v^{l\underline{i}}+\theta(\underline{v}^{l\underline{i}}-v^{l\underline{i}}))(\omega^{l\underline{i}}-\underline{\omega}^{l\underline{i}})d\theta\psi^i(x) dx   \\
\displaystyle\quad \geq -h\sum\limits_{i,j}\int_{x_{j-1}}^{x_{j+1}}\int_0^1\eta_{e\omega}(v^{l\underline{i}}+\theta(\underline{v}^{l\underline{i}}-v^{l\underline{i}}))\big[\frac{N-1}{x}\underline{\varrho}^{l\underline{i}}\\
\displaystyle\qquad\qquad-\frac{\underline{\varrho}^{l\underline{i}}}{x^{N-1}}\int_1^x\underline{\varrho}^{l\underline{i}}(y,t)dy\big]d\theta \psi^i(x) dx -Cl^{\frac{1}{2}}(1+|\log l|),
\end{array}
\end{equation}
for $3\leq\beta\leq4$, where $\psi^i=\psi(x,t_i)$.
From \eqref{eq1208} and the convexity for $\eta_e$, we have
\begin{equation}\label{eq263}
\begin{aligned}
B(\psi)\geq-Cl,  \qquad S(\psi)\geq0,
\end{aligned}
\end{equation}
and
\begin{equation}\label{eq265}
\begin{aligned}
\displaystyle D(\psi)&=\sum\limits_{i,j}\int_{x_{j-1}}^{x_{j+1}}(\eta_e(\bar{v}^{l\underline{i}})-\eta_e(\bar{v}_j^{li}))\psi_j^idx+\sum\limits_{i,j}\int_{x_{j-1}}^{x_{j+1}}(\eta_e(\bar{v}^{l\underline{i}})-\eta_e(\bar{v}_j^{li}))(\psi^i-\psi_j^i)dx\\
& \geq\sum\limits_{i,j}\int_{x_{j-1}}^{x_{j+1}}(\eta_e(\bar{v}^{l\underline{i}})-\eta_e(\bar{v}_j^{li}))(\psi^i-\psi_j^i)dx\\
 &\geq -Cl^{\alpha-\frac{8}{9}},
\end{aligned}
\end{equation}
where $\frac{8}{9}<\alpha<1$ and $\psi_j^i=\psi(x_j,t_i)$.
Thus,
\begin{equation}\label{eq266}
\begin{aligned}
&\displaystyle\iint\limits_{\Pi_T}(\eta_e(\bar{v}^l)\psi_t+q_e(\bar{v}^l)\psi_x)dxdt+\int_1^{+\infty}\eta_e(\underline{v}^l_0(x))\psi(x,0)dx\\
&\quad+h\sum\limits_{i,j}\int_{x_{j-1}}^{x_{j+1}}\int_0^1\eta_{e\omega}(\varrho^{l\underline{i}},\omega_\theta^i)(\frac{N-1}{x}\underline{\varrho}^{l\underline{i}}-\frac{\underline{\varrho}^{l\underline{i}}}{x^{N-1}}\int_1^x\underline{\varrho}^{l\underline{i}}(y,t)dy)d\theta \psi^i(x) dxdt\\
&\geq -Cl^{\frac{1}{2}}(1+|\log l|)-Cl  -Cl^{\alpha-\frac{8}{9}}  \ \ \rightarrow0 \qquad\qquad  \ \ \mbox{as } \ \ l\rightarrow0.
\end{aligned}
\end{equation}
Note that
\begin{equation}\label{eq267}
\begin{aligned}
\underline{v}^l(x,t)\rightarrow v(x,t) ,\ \bar{v}^l(x,t)\rightarrow v(x,t),\ \underline{v}^l_0(x)=\underline{v}^l(x,0)\rightarrow v_0(x)\ \ \mbox{ as } \ \ l\rightarrow0\ \ \mbox{a.e. }
\end{aligned}
\end{equation}
Therefore, using the Jensen's inequality, we have
\begin{equation}\label{eq22267}
\begin{aligned}
&\int_1^{+\infty}\eta_e(\underline{v}^l_0(x))\psi(x,0)dx\\&=\int_1^{+\infty}\eta_e(\underline{v}^l_0(x))\psi(x,0)dx|_{\rho\geq l^\beta}+\int_1^{+\infty}\eta_e(\underline{v}^l_0(x))\psi(x,0)dx|_{\rho< l^\beta}\\
&\leq\int_1^{+\infty}\eta_e(\underline{v}^l_0(x))\psi(x,0)dx+Cl^{\frac{\beta}{2}}\rightarrow \int_1^{+\infty}\eta_e(v_0(x))\psi(x,0)dx \ \ \mbox{ as } \ l\rightarrow0 \ \ \mbox{a.e, }
\end{aligned}
\end{equation}
and
\begin{equation}\label{eq2267}
\begin{aligned}
h&\sum\limits_{i,j}\int_{x_{j-1}}^{x_{j+1}}\int_0^1\eta_{e\omega}(v^{l\underline{i}}+\theta(\underline{v}^{l\underline{i}}-v^{l\underline{i}}))(\frac{N-1}{x}\underline{\varrho}^{l\underline{i}}-\frac{\underline{\varrho}^{l\underline{i}}}{x^{N-1}}\int_1^x\underline{\varrho}^{l\underline{i}}(y,t)dy)d\theta \psi^i(x) dx \\
&\rightarrow \iint\limits_{\Pi_T}\big(\frac{N-1}{x}\varrho-\frac{\varrho}{x^{N-1}}\int_1^x\varrho(s,t)ds\big)\eta_{e\omega}(v)\psi dxdt \ \quad \mbox{ as } \ \ l\rightarrow0\ \ \mbox{a.e. }
\end{aligned}
\end{equation}
Then the inequality \eqref{eq261} holds.

It is easy to check that $(\varrho(x,t),\omega(x,t))$ in \eqref{eq232}, \eqref{eq240} and \eqref{eq261} satisfy the Definition \ref{def1} by \eqref{eq119}.

\end{proof}

\end{document}